\documentclass[10]{amsart} 
\usepackage[utf8]{inputenc} 
\usepackage{enumitem}

\usepackage{geometry} 

\usepackage{graphicx} 
\usepackage{color}
\usepackage{caption}

\usepackage{amsmath,amssymb,amsfonts,mathrsfs}
\usepackage{amsthm,pstricks,enumitem, textcomp, xspace}
\usepackage{todonotes}
\usepackage{enumitem}
\usepackage{dsfont}

\usepackage{fancyhdr} 
\usepackage{hyperref}
\newcommand\myshade{85}
\colorlet{mylinkcolor}{violet}
\colorlet{mycitecolor}{orange}
\colorlet{myurlcolor}{red}

\hypersetup{
  linkcolor  = mylinkcolor!\myshade!black,
  citecolor  = mycitecolor!\myshade!black,
  urlcolor   = myurlcolor!\myshade!black,
  colorlinks = true,
}

\newtheorem{theorem}{Theorem}[section]

\newtheorem{lemma}[theorem]{Lemma}
\newtheorem{corollary}[theorem]{Corollary}
\newtheorem{proposition}[theorem]{Proposition}
\newtheorem{claim}[theorem]{Claim}
\newtheorem{conjecture}[theorem]{Conjecture}

\newtheorem{definition}[theorem]{Definition}
\let\olddefinition\definition
\renewcommand{\definition}{\olddefinition\normalfont}
\newtheorem*{fact}{Fact}
\let\oldfact\fact
\renewcommand{\fact}{\oldfact\normalfont}
\newtheorem{remark}[theorem]{Remark}
\let\oldremark\remark
\renewcommand{\remark}{\oldremark\normalfont}
\newtheorem{example}[theorem]{Example}
\let\oldexample\example
\renewcommand{\example}{\oldexample\normalfont}

\newcommand{\R}{\mathbb R }
\newcommand{\Z}{\mathbb Z }

\newcommand{\sphere}{\mathbb{S}^2}

\newcommand{\Stab}{\mathrm{Stab}} 
\newcommand{\homeo}{\cong} 
\newcommand{\uc}[1]{\tilde{#1}} 

\newcommand{\CC}[1]{\mathbf{#1}} 
\newcommand{\CCv}[1]{\mathbf{#1}} 
\newcommand{\CCc}[1]{\mathbf{#1}} 

\newcommand{\Hyp}[1]{\hat{\mathcal{#1}}} 
\newcommand{\hyp}[1]{\hat{\mathfrak{#1}}} 

\newcommand{\Hs}[1]{\mathcal{#1}} 
\newcommand{\hs}[1]{\mathfrak{#1}} 
\newcommand{\comp}[1]{{#1}^*} 

\newcommand{\bdry}{\partial}
\newcommand{\bhyp}[1]{\bdry\hat{\mathfrak{#1}}} 
\newcommand{\bhs}[1]{\bdry\mathfrak{#1}} 
\newcommand{\obhs}[1]{\bdry^\circ\mathfrak{#1}} 
\newcommand{\grid}[1]{\Gamma}
\newcommand{\gridy}[1]{\mathbf{#1}}
\newcommand{\charf}[1]{\mathds{1}_{#1}}
\newcommand{\inv}[1]{J^{-1}(#1)}

\DeclareMathOperator{\type}{type}
\DeclareMathOperator{\diam}{diam}
\DeclareMathOperator{\Link}{Link}

\title{Sphere boundaries of hyperbolic groups}
\author{Benjamin Beeker}

\address{Department of Mathematics, Hebrew University, Jerusalem, Israel}
\email{beeker@tx.technion.ac.il}
\urladdr{}

\author{Nir Lazarovich}
\address{Department of Mathematics, ETH Z\"urich, R\"amistrasse 101, 8092 Z\"urich, Switzerland}
\email{nir.lazarovich@math.ethz.ch}

\date{} 

\begin{document}
\maketitle

\begin{abstract}
We show that a one-ended simply connected at infinity hyperbolic group $G$ with enough codimension-1 surface subgroups has $\bdry G \homeo \sphere$. By Markovic \cite{Mar13}, our result gives a new characterization of virtually fundamental groups of hyperbolic 3-manifolds.
\end{abstract}

\section{Introduction}
We recall the following well-known conjecture.

\begin{conjecture}[Cannon's Conjecture \cite{Can91}] Let $G$ be a hyperbolic group. If $\partial G \homeo \sphere$ then $G$ acts geometrically on the hyperbolic space $\mathbb{H}^3$.
\end{conjecture}

In \cite{Mar13}, Markovic described the following criterion under which the conjecture is true. 

\begin{theorem}
Let $G$ be a hyperbolic group that acts faithfully on its boundary $\partial G = \sphere$ and contains enough quasiconvex surface subgroups. Then $G$ acts geometrically on the hyperbolic space $\mathbb{H}^3$.
\end{theorem}

In \cite{KaMa12}, Kahn and Markovic showed that the fundamental group of a hyperbolic 3-manifold contains enough quasiconvex surface subgroups. This shows that Markovic's criterion is also necessary. 
More about the history of Cannon's conjecture and the works preceding Markovic's criterion can be found in a survey about boundaries of hyperbolic groups by Kapovich and Benakli \cite{KaBe02}.

In this paper we prove that it is possible to replace the assumption that the boundary at infinity is homeomorphic to $\sphere$ by the assumption of vanishing of the first cohomology of $G$ at infinity (see Subsection \ref{subsec:properties at infinity} for the definitions). In other words, we have the following result.

\begin{theorem}[Main result] \label{main result}
Let $G$ be a one-ended hyperbolic group. Assume that $G$ has vanishing first cohomology over $\Z/2$ at infinity, and that $G$ contains enough quasiconvex codimension-1 surface subgroups. Then $\partial G \homeo\sphere$. 
\end{theorem}

Combining this result with Markovic \cite{Mar13} and Kahn-Markovic \cite{KaMa12} we get the following.

\begin{corollary}
Let $G$ be a hyperbolic group. The following are equivalent:
\begin{enumerate}
\item $G$ acts geometrically on $\mathbb{H}^3$.
\item $\partial G \homeo \sphere$, and $G$ contains enough quasiconvex surface subgroups.
\item $G$ is one-ended, has vanishing first cohomology over $\Z/2$ at infinity and contains enough quasiconvex codimension-1 surface subgroups.
\end{enumerate}
\end{corollary}

The main tool we use is the Kline Sphere Characterization which was proven by Bing in \cite{Bin46}.

\begin{theorem}[The Kline Sphere Characterization Theorem] \label{KSCT} Let $X$ be a topological space which is
\begin{enumerate}
\item \label{KSCT_elementary}nondegenerate, metrizable, compact, connected, locally connected,
\item \label{KSCT_pair_of_points}not separated by any pair of points, and
\item \label{KSCT_Jordan}separated by any Jordan curve.
\end{enumerate} 
Then $X\homeo\sphere$.
\end{theorem}

The outline of the proof of the main result, and of the paper, is as follows.
In Section \ref{sec: preliminaries}, we provide the necessary preliminaries for the proof.
We end Section \ref{sec: preliminaries} with a summary of the additional assumptions one can make on the group $G$.
In Lemma \ref{lem: non-separation} in Section \ref{sec: no separating pair of points}, we prove that no pair of points separates the boundary of such a group. 
In Sections \ref{sec: outline}--\ref{sec: Jordan thm}, we prove that any Jordan curve separates, using ideas from Jordan's original proof of the Jordan Curve Theorem which we outline in Section \ref{sec: outline} (see Hales \cite{Hal07} for details).

We finish the introduction with an example of a one-ended hyperbolic group that has enough quasiconvex codimension-1 surface subgroups and whose boundary is not a sphere.

Let $L$ be the flag complex triangulation of the 2-dimensional torus consisting of 96 triangles illustrated in Figure \ref{fig: counterexample}, where the edges of the hexagon labeled by $a$, $b$ and $c$ are glued accordingly to obtain a torus. 
Let $G$ be the right-angled Coxeter group associated with $L$. That is, $G$ is the group given by the following presentation.
\[ G=\left<  s\in L^{(0)} \middle| s^2=1, \forall s\in L;\; [s,s']=1, \forall \{s,s'\}\in L^{(1)} \right> \]

\begin{figure}
\begin{center}
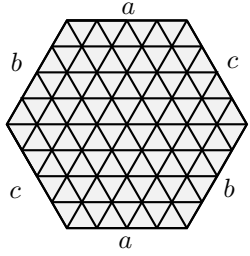
\caption{The torus link of the counterexample}
\label{fig: counterexample}
\end{center}
\end{figure}

The group $G$ acts properly and cocompactly on the Davis complex $\CC{X}$ associated with $L$, which happens to be the unique CAT(0) cube complex whose link is isomorphic to $L$ at each vertex (See \cite{Laz14}). 

The group $G$ is hyperbolic since $L$ has no isometrically embedded geodesic of length $2\pi$ (when considered with the spherical metric), and it is one-ended since the link $L$ does not have a separating simplex.
The hyperplane stabilizers are isomorphic to the right-angled Coxeter groups associated with the link of a vertex in $L$, which are 6-cycle graphs in our case. 
Thus, the hyperplane stabilizers are Fuchsian groups, which implies that $G$ has enough codimension-1 surface subgroups.

Let $p$ be the projection map from $\bdry\CC{X}$ to the link $\Link(\CCv{x},\CC{X})$ of a vertex $\CCv{x}\in\CC{X}$, that assigns to each boundary point $\xi$ the direction in $\Link(\CCv{x},\CC{X})$ of the unique geodesic that connects $\CCv{x}$ to $\xi$. Since $\CC{X}$ has extendable geodesics one can lift any curve on $\Link(\CCv{x})$ to the boundary $\bdry\CC{X}$. Therefore, the induced map $\pi_1(\bdry\CC{X}) \to \pi_1 (\Link(\CCv{x},\CC{X}))\simeq\Z^2$ is onto. Hence, the boundary $\bdry G$ is not homeomorphic to $\sphere$.

\paragraph*{Acknowledgements.}
The authors would like to thank Michah Sageev for his support and helpful comments. We also thank Antoine Clais and Jason Manning for finding a mistake in a previous version of this paper. 

\section{Preliminaries}
\label{sec: preliminaries}

We begin by a survey of definition and results concerning CAT(0) cube complexes and quasiconvex subgroups of hyperbolic groups. For a more complete survey see, for example, Sageev \cite{Saa12}.

\subsection{CAT(0) cube complexes and hyperplanes}

\begin{definition}[CAT(0) cube complexes]
A \emph{cube complex} is a complex made by gluing unit Euclidean cubes (of varying dimension) along their faces using isometries.
A cube complex is \emph{CAT(0)} if it is CAT(0) with respect to the quotient metric induced by endowing each cube with the Euclidean metric (See \cite{BrHa99}).
\end{definition}

As was first observed by Sageev \cite{Sag95}, CAT(0) cube complexes naturally carry a combinatorial structure given by the associated hyperplanes and halfspaces.
We now recall their definition and properties.

\begin{definition}[hyperplanes]
Let $\CC{X}$ be a CAT(0) cube complex.
The equivalence relation on the edges of $\CC{X}$ generated by $\CCc{e}\sim\CCc{e}'$ when $\CCc{e}$ and $\CCc{e}'$ are parallel edges in a square of $\CC{X}$ is called the parallelism relation.
The equivalence classes of edges under the parallelism relation are the \emph{combinatorial hyperplanes} of $\CC{X}$.
The convex hull of the midpoints of the edges of a combinatorial hyperplane is called a \emph{hyperplane}.
We denote the set of hyperplanes in $\CC{X}$ by $\Hyp{H}=\Hyp{H}(\CC{X})$.
\end{definition}

The main features of hyperplanes are summed in the following. 

\begin{proposition}
Let $\CC{X}$ be a CAT(0) cube complex. Then every hyperplane $\hyp{h}\in \Hyp{H}$ is naturally a CAT(0) cube complex of codimension-1 in $\CC{X}$, and $\CC{X}\setminus \hyp{h}$ has exactly two components.
\end{proposition}

\begin{definition}
The components of $\CC{X}\setminus\hyp{h}$ are the \emph{halfspaces} of $\CC{X}$ associated to (or bounded by) $\hyp{h}$.
The set of halfspaces is denoted by $\Hs{H}=\Hs{H}(\CC{X})$.
There is a natural map $\hat{}:\Hs{H}\to\Hyp{H}$ which maps each halfspace to its bounding hyperplane.
The set of halfspaces also carries a natural complementation involution $\comp{}:\Hs{H}\to\Hs{H}$ which maps each halfspace $\hs{h}$ to the other component of $\CC{X}\setminus\hyp{h}$.
\end{definition}

\subsection{Cubulating hyperbolic groups}

Recall the following definitions about quasiconvex subgroups of Gromov hyperbolic groups.

\begin{definition}
A quasiconvex subgroup $H$ of a hyperbolic group $G$ is \emph{codimension-1} if $G/H$ has more than one end.
The group $G$ has \emph{enough codimension-1 subgroups} if every two distinct points in $\bdry{G}$ can be separated by the limit set of a codimension-1 quasiconvex subgroup.
\end{definition}

By results of Sageev \cite{Sag95}, Gitik-Mitra-Rips-Sageev \cite{GMRS98} and Bergeron-Wise \cite{BeWi12} we have the following.

\begin{theorem}
Let $G$ be a hyperbolic group with enough codimension-1 subgroups, then $G$ acts properly cocompactly on a finite dimensional CAT(0) cube complex whose hyperplane stabilizers belong to the family of codimension-1 subgroups.
\end{theorem}

This theorem is the starting point for our proof, since the group $G$ in Theorem \ref{main result} is assumed to have enough codimension-1 surfaces subgroups, it therefore acts properly and cocompactly on a finite dimensional CAT(0) cube complex $\CC{X}$ with surface group hyperplane stabilizers.
The proof of the main theorem is then based on understanding how the limit sets of hyperplanes interact and how they determine the topology of the boundary of $G$.

Recall that any quasiconvex subgroup $H$ in $G$ is itself a hyperbolic group and there is a well-defined homeomorphic embedding $\bdry{H}\to\bdry{G}$ whose image is the limit set of $H$ in $\bdry{G}$.
In particular, in our situation, the stabilizer of each hyperplane $\hyp{h}$ in $\CC{X}$ acts properly cocompactly on the hyperplane and thus is a quasiconvex surface subgroup. Its boundary is therefore a circle, which embeds as the limit set of $\hyp{h}$ in $\bdry{G}=\bdry\CC{X}$. We denote this limit set by $\bhyp{h}$.

However, for halfspaces we have to distinguish between two kinds of limit sets: the \emph{closed limit set} $\bhs{h}$ and the \emph{open limit set} $\obhs{h}$. 
The former is simply the limit set of $\hs{h}$ as a subset of $\CC{X}$, while the latter is the set of all geodesic rays in $\bhs{h}$ which do not stay at a bounded distance from $\hyp{h}$. Equivalently, we have $\bhs{h}=\bhyp{h}\sqcup\obhs{h}$.
As their names suggest the open (resp. closed) limit sets are indeed open (resp. closed) subsets of $\bdry G$. Moreover, we have the following.

\begin{lemma}\label{basis for topology}
The open limit sets of halfspaces form a basis for the topology of $\bdry G = \bdry \CC{X}$.
\end{lemma}

\begin{proof}
Let $\CCv{x}_0$ be a fixed vertex in $\CC{X}$.
Identify $\bdry\CC{X}$ with the visual Gromov boundary from $\CCv{x}_0$.
Let $\xi\in\bdry\CC{X}$, and let $\gamma$ be a geodesic ray such that $\gamma(\infty)=\xi$ and $\gamma(0)=\CCv{x}_0$.
Let $\Hs{H}_\gamma$ be the set of halfspaces $\hs{h}$ such that $\gamma\setminus\hs{h}$ is bounded.
The set $\Hyp{H}_\gamma$ is infinite and does not contain pairs of disjoint halfspaces.
Since $\CC{X}$ is finite dimensional, $\Hyp{H}_\gamma$ must contain an infinite descending chain of halfspaces $\hs{h}_1\supset \hs{h}_2 \supset\ldots$.
Moreover $d(\CCv{x}_0,\hs{h}_n)\to\infty$, and since halfspaces are convex, it follows that $\diam(\bhs{h}_n)\to 0$ (with respect to a visual metric on the boundary).
Moreover, since $\hs{h}_n$ form a descending chain of halfspaces that cross $\gamma$ it follows that $\gamma$ does not remain within a bounded distance from the bounding hyperplanes $\hyp{h}_n$ for all $n$.
Thus, $\xi\in\obhs{h}_n$ and $\obhs{h}_n$ form a local basis at $\xi$.
\end{proof}

Recall from Caprace-Sageev \cite{CaSa11} that a cube complex is \emph{essential} if there is no halfspace which is at bounded distance from its bounding hyperplane.
The rank rigidity results in \cite{CaSa11} imply that if $G$ is a hyperbolic group that acts properly and cocompactly on a CAT(0) cube complex $\CC{X}$, then there is a convex $G$-invariant subcomplex $\CC{Y}\subseteq\CC{X}$ which is essential, and the open limit set of every halfspace is non-empty.
For that reason, we assume from now on that $\CC{X}$ is essential. 

\subsection{Properties of groups at infinity}
\label{subsec:properties at infinity}

In this section we define topological properties of spaces and groups at infinity.
We begin with the following general definition.

\begin{definition}
Let $F$ be a contravariant (resp. covariant) functor from the category of topological spaces $\mathbf{Top}$ to the category of groups $\mathbf{Grp}$.
Then, Let $F_\infty$ (resp. $F^\infty$) be the functor from the category of non-compact topological spaces (with proper continuous maps as morphisms) that assigns to a space $X$ the direct limit (resp. inverse limit) of the directed set $F(X\setminus K)$ where $K$ ranges over the compact subsets in $X$.
\end{definition}

Applying the above definition to the contravariant functor $H^1(\cdot\; ;R)$, we say that $X$ has \emph{vanishing first  cohomology over a ring $R$ at infinity} if for every compact set $K\subset X$ and every 1-cocycle $\alpha$ on $X\setminus K$ there exists a bigger compact set $K'$ such that $\alpha$ restricted to $X\setminus K'$ is a coboundary.

Similar definitions can be defined by applying the above definition to $\pi_1$, $H_n (\cdot\; ;R)$ and $H^n (\cdot\; ;R)$. We recall in particular the following well-studied notion.

A topological space $X$ is said to be \emph{simply connected at infinity} if for every compact set $K$ there exists a compact set $K'\supseteq K$ such that any loop in $X\setminus K'$ is null homotopic in $X\setminus K$, i.e, if the map $\pi_1 (X\setminus K')\to \pi_1 (X\setminus K)$ is trivial.
%

We remark that any space which is simply connected at infinity has, in particular, vanishing first cohomology over $\Z/2$.

\begin{definition}

Let $G$ be a group of type F (i.e, has a compact $K(G,1)$), and let $X$ be the universal cover of a compact $K(G,1)$ of $G$. We say that $G$ has \emph{vanishing first cohomology over a ring $R$ at infinity} (resp. is \emph{simply connected at infinity}) if $X$ has vanishing first cohomology over a ring $R$ at infinity (resp. is simply connected at infinity).

More generally, if $G$ has a finite index subgroup $H$ of type F, then we say that $G$ has \emph{vanishing first cohomology over a ring $R$ at infinity} (resp. is \emph{simply connected at infinity}) if $H$ has vanishing first cohomology over a ring $R$ at infinity (resp. is simply connected at infinity).
\end{definition}

In \cite{Bri93}, Brick showed that being finitely presented at infinity is a property of finitely presented groups, which is a quasi-isometric invariant. 

\begin{lemma}
The notions defined above do not depend on the compact $K(G,1)$ and the choice of a finite index subgroup.
\end{lemma}

\begin{proof}
Let $X$ and $Y$ be two compact $K(G,1)$. 
Then there exists a $G$-equivariant homotopy equivalence $f:\uc{X}\to \uc{Y}$ which is proper.
We prove it for the $n$-th cohomology functor.
Since this map is proper it induces a map $H^n(Y\setminus K) \to H^n(X\setminus f^{-1}(K)) \to  H^n_\infty (X)$, and thus, a map $H^n_\infty(f):H^n_\infty(Y)\to H^n_\infty(X)$.
Similarly, the $G$-equivariant homotopy inverse map $g:\uc{Y}\to \uc{X}$ induces a map $H^n_\infty(g):H^n_\infty(X)\to H^n_\infty(Y)$.
The composition $f\circ g$ induces a map $H^n_ \infty (f)\circ H^n_\infty (g):H^n_ \infty (X)\to H^n_\infty (X)$. Since $f\circ g$ is homotopic to the identity by an homotopy of $G$-equivariant proper maps, it follows that the induced map in $H^n_\infty$ is the identity map.  Similarly $H^n_\infty(g)\circ H^n_\infty(f)$ is the identity map.

Let $H_1,H_2$ be two finite-index type F subgroups of $G$. Then $H_1\cap H_2$ is also a finite-index type F subgroup of $G$.
Therefore, one can assume $H_1\le H_2 \le G$. But in this case any compact $K(G,1)$ for $H_2$ has a finite cover which is a $K(G,1)$ of $H_1$. This completes the proof since the defined notions only depend on the universal cover of the finite $K(G,1)$.
\end{proof}

Since every torsion-free hyperbolic group $G$ is of type F (for example, its Rips complex for sufficiently large parameter $r$ is the universal cover of a compact $K(G,1)$), the notions defined above can be defined for every torsion-free hyperbolic group.
It is a very well known question whether hyperbolic groups are virtually torsion free. However, in the context of Theorem \ref{main result}, $G$ is a hyperbolic group that acts properly cocompactly on a CAT(0) cube complex. 
It follows from works of Haglund-Wise \cite{HaWi08} and Agol \cite{Ago13} that $G$ is virtually torsion-free.
Therefore, the assumption that $G$ has vanishing first cohomology over $\Z/2$ at infinity is well-defined for $G$.
By replacing $G$ by its finite-index torsion-free subgroup it is enough to prove the theorem for torsion-free groups. Hence, in the remainder of the paper we assume that $G$ is torsion-free.

We remark that since $G$ is torsion-free and acts properly and cocompactly on the CAT(0) cube complex $\CC{X}$ it follows that $\CC{X}$ is the universal cover of a compact $K(G,1)$ (namely $\CC{X}/G$), and thus, by assumption, has vanishing first cohomology over $\Z/2$ at infinity.

\subsection{Summary of preliminaries}
\label{subsec: sum of prelims}

In the previous subsections we have seen that under the assumptions of Theorem \ref{main result}, one can make further assumptions on $G$. 
In this subsection we collect the assumptions we made on $G$ that will be used in the remainder of the paper:
\begin{itemize}
\item The group $G$ is hyperbolic, one-ended and torsion-free.
\item There exists a finite dimensional CAT(0) cube complex $\CC{X}$ such that:
\begin{itemize}
\item The group $G$ acts freely and cocompactly on $\CC{X}$.
\item The hyperplane stabilizers are surface subgroups.
\item The cube complex $\CC{X}$ is essential, and in particular, open limit sets of halfspaces are non-empty.
\item The cube complex $\CC{X}$ has vanishing first cohomology at infinity.
\end{itemize}
\end{itemize}
\section{Connectivity and non-separation by a pair of points}
\label{sec: no separating pair of points}

In this section we prove that no pair of points can separate the boundary of a group that satisfies the assumption of the main theorem (Theorem \ref{main result}).

Recall the following result of Bowditch \cite{Bow98}

\begin{theorem} \label{Bowditch}
Let $G$ be a one-ended word-hyperbolic group such that $\partial G \ne \mathbb{S}^1$, Then:
\begin{enumerate}
\item \label{Bowditch_loc_conn} The boundary is locally connected and has no global cut points.
\item \label{Bowditch_loc_cutpt} The following are equivalent:
\begin{enumerate}
\item The group $G$ does not split essentially over a two-ended subgroup 
\item The boundary $\partial G$ has no local cut point. 
\item The boundary $\partial G$ is not separated by a pair of points.
\item The boundary $\partial G$ is not separated by a finite set of points.
\end{enumerate}
\end{enumerate}
\end{theorem}

The following lemmas and corollaries apply to the setting of Theorem \ref{main result}, since we assume $G$ has enough codimension-1 quasiconvex surface subgroups. However, we chose to phrase them in a more general setting.

\begin{lemma}
\label{lem: non-separation}
Let $G$ be a one-ended hyperbolic group that contains enough codimension-1 quasi\-convex one-ended subgroups. Then $\partial G$ is not separated by any pair of points.
\end{lemma}
\begin{proof}
Let $\xi,\zeta\in \partial G$ be distinct points. Let $\hyp{h}$ be a hyperplane such that $\xi\in\obhs{h}$ and $\zeta\in\obhs{\comp{h}}$. 
Let $A=\bhs{h}$ and $B=\bhs{\comp{h}}$.
The closed sets $A,B$ satisfy $(A\setminus \{\xi\}) \cup B = \partial G \setminus \{\xi\}$ and $(A\setminus \{\xi\})\cap B = \bhyp{h}$. By \ref{Bowditch_loc_conn} of Theorem \ref{Bowditch} we know that $\partial G \setminus \{\xi\}$ is connected, and thus each of the two halves $A\setminus \{\xi\}$ and $B$ is connected (since we assumed $\bhyp{h}$ is connected). Similarly, $A$ and $B\setminus \{\zeta\}$ are connected. Finally, $\partial G \setminus \{\xi,\zeta\}= (A\setminus \{\xi\}) \cup (B\setminus \{\zeta\})$ is connected as union of connected intersecting sets.
\end{proof}

By \ref{Bowditch_loc_cutpt} of Theorem \ref{Bowditch} we have the following.

\begin{corollary}\label{cor: no local cutpoints}
Let $G$ be a one-ended hyperbolic group that contains enough codimension-1 quasi\-convex one-ended subgroups. Then $\partial G$ does not have any local cutpoints and is not separated by any finite number of points.\qed
\end{corollary}

We continue our discussion with connectivity properties of $\bdry{G}$.

\begin{lemma}\label{lem: bdry is path connected}
Let $G$ be one-ended hyperbolic group that contains enough codimension-1 quasi\-convex subgroups with path connected boundaries, then $\bdry{G}$ is path-connected and locally path-connected.
\end{lemma}

\begin{proof}
Let $\CC{X}$ be the associated CAT(0) cube complex. 
Let $\{\bhs{h}_1,\ldots,\bhs{h}_n\}$ be a cover of $\bdry{\CC{X}}$ which is minimal, \textit{i.e.} which has no proper subcover, with closed limit sets of halfspaces with $n>2$.

Since the boundaries of hyperplanes are connected, the set $\bigcup_{i=1}^n \bhyp{h}_i$ is connected. Otherwise, there are two subsets of the cover, which we assume are $\{\bhs{h}_1,\ldots,\bhs{h}_k\}$ and $\{\bhs{h}_{k+1},\ldots,\bhs{h}_n\}$ such that $\bigcup_{i=1}^k \bhyp{h}_i$ and  $\bigcup_{i=k+1}^n \bhyp{h}_i$ are disjoint. But this would contradict the assumption that $\{\bhs{h}_1,\ldots,\bhs{h}_n\}$ is a minimal cover.
By assumption, each $\bhyp{h}_i$ is path-connected and hence $\bigcup_{i=1}^n \bhyp{h}_i$ is path-connected.

In particular, if $\bhs{h}_1,\bhs{h}_2$ are disjoint closed limit sets of halfspaces, then by completing them to a minimal cover (for example by adding limit sets of halfspaces of much smaller diameter), we conclude that any point in $\bhyp{h}_1$ can be connected to any point in $\bhyp{h}_2$. Moreover, one can do so by a path in $\bdry{\CC{X}}\setminus (\obhs{h}_1\cup\obhs{h}_2)$.

Let $\xi,\zeta$ be two distinct points in $\bdry{\CC{X}}$, and let $\obhs{h}_n$ and $\obhs{k}_n$ be descending local bases for $\xi$ and $\zeta$ respectively. We may assume that $\bhs{h}_1$ and $\bhs{k}_1$ are disjoint. 
By the previous discussion we can connect each $\bhyp{h}_i$ to $\bhyp{h}_{i+1}$ (resp. $\bhyp{k}_i$ to $\bhyp{k}_{i+1}$)  with a path which stays between $\bhyp{h}_i$ and $\bhyp{h}_{i+1}$ (resp. $\bhyp{k}_i$ and $\bhyp{k}_{i+1}$).
We can also connect $\bhyp{h}_1$ and $\bhyp{k}_1$.
By concatenating the above paths, with a suitable parameterization, one can show that there exists a path that connects $\xi$ to $\zeta$. Thus, $\CC{X}$ is path connected.

Finally, since $\CC{X}=\bhs{h}\cup\bhs{\comp{h}}$ and $\bhs{h}\cap\bhs{\comp{h}}=\bhyp{h}$ are both path connected. Each closed limit set of halfspace $\bhs{h}$ is path connected. This implies that the space is locally path-connected.
\end{proof}

The following lemma is a direct corollary of the above, and will be used later on in Lemma \ref{lem: parity of Jordan curve}.

\begin{lemma}[No Blob Lemma]\label{lem: no blob}
Let $X$ be a compact, path connected, locally path connected metric space with no cutpoints.
Then for all $\epsilon>0$ there exists $\delta>0$ such that for all $x\in X$ and for all $A\subset B(x,\delta)$ there is at most one component of $X\setminus A$ which is not contained in $B(x,\epsilon)$.
\end{lemma}

\begin{proof}
Assume for contradiction that there is $\epsilon>0$ such that for all $n$ there is a point $x_n\in X$ and subset $A_n \subset B(x_n,\frac{1}{n})$ and two distinct components $B_n,C_n$ of $X\setminus A_n$ which have points $b_n,c_n$ respectively outside $B(x,\epsilon)$. By passing to a subsequence if necessary we may assume that the sequences $x_n,b_n,c_n$ converge to $x,b,c$ respectively. Since  $d(b_n,x_n)>\epsilon$ (resp. $d(c_n,x_n)>\epsilon$) we deduce that $d(b,x)\ge\epsilon$ (resp. $d(c,x)\ge\epsilon$).

The space $X$ has no cut point, therefore we can connect $b$ and $c$ with a path $\gamma$ that does not pass through $x$. Since $A_n \to x$, for $n$ big enough $A_n$ is disjoint from $\gamma$. By local path connectivity, for $n$ big enough $b_n$ and $b$ (resp. $c$ and $c_n$) can be connected by a short path $\gamma_b$ (reps. $\gamma_c$) that avoids $A_n$. The concatenated path $\gamma_b*\gamma*\gamma_c$ from $b_n$ to $c_n$ avoids $A_n$ and thus contradicts the assumption that $B_n$ and $C_n$ are distinct components of $X\setminus A_n$.
\end{proof}

\section{Outline of the proof of Jordan's Theorem}
\label{sec: outline}

The aim of this section is to provide a short outline of the proof that any Jordan curve on $\bdry G$ separates. From this section on, let $G$ and $\CC{X}$ be a group and a CAT(0) cube complex satisfying the assumptions made in Subsection \ref{subsec: sum of prelims}.

We first recall Jordan's original proof of his theorem. 
Let $P$ be a polynomial curve. One can associate a parity function to $P$ which, for a point $x\in \R^2\setminus P$ and a generic ray $l$ starting at $x$, counts the number of intersections of $l$ with $P$ mod $2$. 
It is easy to see that this parity function is constant on the connected components of $\R^2\setminus P$. 
Now, for a Jordan curve $J$, one approximates $J$ with polygonal curves $P_n$, and proves that the sequence of the parity functions of the polygonal curves $P_n$ has a limit which is constant on connected components of $\R^2\setminus P$ (and does not depend on the choice of approximating polygonal curves). 
The final crucial step is to show that this parity function indeed obtains two values, which shows that $\R^2\setminus P$ has at least two components.
A more subtle point in the proof of Jordan's Theorem is the proof that each of the two parity regions is connected, but luckily we will not need to prove this since we are only interested in showing that the Jordan curve separates.
The details of Jordan's proof of his theorem can be found in a paper by Hales \cite{Hal07}.

The main idea in our proof is to replace polygonal path approximations with ``piecewise hyperplane paths'' (or ``PH paths''). 
Those are paths in $\bdry\CC{X}$ which are piecewise subsegments of limit sets of hyperplanes (which are homeomorphic to circles).  In Section \ref{sec: PH curves} we give a precise definition and show that every Jordan curve can be approximated by PH curves.

In order to obtain a parity function we first approximate the PH paths by ``grids''.
Those are collection of hyperplanes which are connected along ``connectors''.
Precise definitions and examples are given in Subsections \ref{subsec: connectors} and \ref{subsec: grids}. 
In Section \ref{sec: constructing cocycles} we describe how, given   an (oriented) grid, one can assign a 1-cocycle (over $\Z/2$) which is defined in $\CC{X}$ outside a large enough ball.
Using the vanishing of the first cohomology over $\Z_2$ at infinity, one can find a coboundary for this cocycle, which extends to a map which is defined on the boundary except at the grid. We call it ``the parity function of the grid''. Its construction and properties are discussed in Section \ref{sec: parity of grid}.

In Section \ref{sec: the parity of a Jordan curve}, we show how one obtains a parity function of a PH curve by taking the limit of the parity functions of approximating grids. 
We then define the parity function for a Jordan curve, by taking the limit of the parities of a sequence of approximating PH curves.

The crucial part, as in Jordan's proof, is to show that the limiting parity obtains two values.
This is done in the final section, Section \ref{sec: Jordan thm}, by analyzing how (certain) PH curves intersect a given hyperplane.
This is strongly based on Proposition \ref{prop: two state solution} which analyzes the ways in which two hyperplanes can intersect.

\section{Intersections of hyperplanes}
\label{sec: intersection}

Since our proofs involves studying intersections of hyperplanes which have surface group stabilizers, we begin by recalling the following well-known elementary fact about surface groups. 

\begin{fact}
Finitely generated subgroups of surface groups are one of the following: trivial, cyclic, free, or of finite index.
\end{fact}

Since our hyperplane stabilizers are surface subgroups we may assume that any two intersecting hyperplanes do so on a cyclic, free, finite index or trivial subgroup. 
This means that the intersection of their limit sets at infinity is either empty, a pair of points, a Cantor set, or the whole circle.

\subsection{Connectors}
\label{subsec: connectors}

Let $\hyp{h},\hyp{k}$ be two intersecting hyperplanes. A \emph{connector} is a pair $(C,\{\hyp{h},\hyp{k}\})$, which by abuse of notation we will simply denote by $C$, where $C$ is a non-empty clopen proper subset $\emptyset\ne C\subsetneq\bhyp{h}\cap\bhyp{k}$. In particular,  $\bhyp{h}\cap\bhyp{k}$ is not the whole circle. We say that the connector $C$ is \emph{supported on the hyperplanes} $\hyp{h},\hyp{k}$.

Two connectors $(C_1,\{\hyp{h}_1,\hyp{k}_1\}),(C_2,\{\hyp{h}_2,\hyp{k}_2\})$ are \emph{disjoint} if either $\{\hyp{h}_1,\hyp{k}_1\}\cap\{\hyp{h}_2,\hyp{k}_2\}=\emptyset$ or $C_1\cap C_2=\emptyset$.

Let $I_1,I_2,\ldots, I_n$ be disjoint intervals in $\bhyp{h}$ whose endpoints are disjoint from $\bhyp{h}\cap\bhyp{k}$ and such that $C = \bhyp{h}\cap\bhyp{k} \cap \bigcup _j I_j $. To each interval $I_j$ we assign the number $\type_{\hyp{h}}(I_j)=1 \in \Z/2$ if the two endpoints of $I_j$ are in different sides of $\hyp{k}$, and $\type_{\hyp{h}}(I_j)=0$ otherwise. 

Next, we extend this definition to the connector $C$ by assigning its type in $\hyp{h}$ to be \[ \type_{\hyp{h}}(C) = \sum _{j=1} ^{n} \type_{\hyp{h}}(I_j) \in \Z/2. \] 
It is easy to verify that this definition depends only on $C,\hyp{h},\hyp{k}$ and does not depend on the choice of the open cover by intervals $I_1,\ldots I_n$ (by passing to a common subdivision of the two open covers).

Finally, we define the type of $C$ by $\type(C,\hyp{h},\hyp{k}) = (\type_{\hyp{h}}(C),\type_{\hyp{k}}(C))$.

We will later show in Proposition \ref{prop: two state solution} that the only possible types, under our assumptions, are $(0,0)$ and $(1,1)$.

\subsection{Grids and oriented grids}
\label{subsec: grids}

Let $\hyp{h}$ be a hyperplane, and let $\hyp{k}_1,\ldots,\hyp{k}_n$ be a collection of (not-necessarily distinct) hyperplanes which intersect $\hyp{h}$. Let $\gridy{C}_{\hyp{h}} = \{C_1,\ldots,C_n\}$ be a set of disjoint connectors such that each $C_i$ is supported on $\hyp{h},\hyp{k}_i$. The choice of $\gridy {C}_{\hyp{h}}$ is \emph{admissible} if $\sum_i \type_{\hyp{h}} (C_i) = 0$.

\begin{remark}\label{rmk: Examples of admissible hyperplanes} The following examples show how this definition behaves for $n=0,1,2$:
\begin{itemize}
\item If $n=0$, then $\gridy {C}_{\hyp{h}}=\emptyset$, and it is admissible. 
\item If $n=1$, $\gridy {C}_{\hyp{h}}=\{C_1\}$ is admissible if and only if $\type_{\hyp{h}}(C_1)=0$.
\item If $n=2$, $\gridy {C}_{\hyp{h}}=\{C_1,C_2\}$ is admissible if and only if $\type_{\hyp{h}}(C_1)=\type_{\hyp{h}}(C_2)$.
\end{itemize}
\end{remark}

We comment at this point that admissible sets of connectors correspond to trivial classes in the first cohomology at infinity of $\hyp{h}$ over $\Z/2$ (which is isomorphic to $\Z/2$), and as such they have a coboundary (in fact, they have exactly two coboundaries). We call its extension to the boundary an  ``orientation'' for $\hyp{h}$. We make this comment explicit, avoiding the use of cohomology at infinity, in the following definition and claim.

\begin{definition}
\label{def: orientation of hyp} 
Let $\hyp{h}$ be a hyperplane and let $\gridy {C}_{\hyp{h}}$ be an admissible collection of disjoint connectors on $\hyp{h}$. 
A function $\alpha_{\bhyp{h}}: \bhyp{h}\setminus \bigcup \gridy {C}_{\hyp{h}} \to \Z/2$ is called an \emph{orientation} if for each $\xi\in\bhyp{h}$ there exists an open neighborhood $V$ of $\xi$ in $\bhyp{h}$ such that either
\begin{itemize}
\item $V$ intersects exactly one connector $C\in \gridy{C}_{\hyp{h}}$ which is supported on $\hyp{h}$ and another hyperplane $\hyp{k}$, and the function $\alpha_{\bhyp{h}}$ on $V\setminus C$ is the characteristic function of $\bhs{k}$ of an orientation $\hs{k}$ of the hyperplane $\hyp{k}$ (which depends on $V$), or
\item $V$ does not intersect any connector and $\alpha_{\bhyp{h}}$ is constant on $V$.
\end{itemize}
\end{definition}

\begin{claim}
\label{claim: orientation of grids} 
Given a hyperplane $\hyp{h}$ and an admissible collection of disjoint connectors $\gridy {C}_{\hyp{h}}$, there are exactly two orientations $\alpha_{\bhyp{h}}: \bhyp{h}\setminus \bigcup \gridy {C}_{\hyp{h}} \to \Z/2$, and they differ by the constant function $1$. 
\end{claim}

\begin{proof}
Since $\gridy {C}_{\hyp{h}}$ is a finite set of disjoint closed subsets of $\bhyp{h}$, one can find a partition $I_1,\ldots,I_n$ of the circle $\bhyp{h}$ into subintervals whose endpoints $\xi_0,\xi_1,\ldots,\xi_n=\xi_0$ are disjoint from the connectors $\gridy {C}_{\hyp{h}}$, and each interval intersects at most one connector.
Take $\xi_0$, the joint endpoint of $I_1$ and $I_n$, and define $\alpha_{\bhyp{h}}(\xi_0)\in\Z/2$ arbitrarily. We show that this choice defines $\alpha_{\bhyp{h}}$ uniquely. 

For each interval $I_m$, if $I_m$ intersect a connector $C$ then we define $\type_{\hyp{h}} (I_m)$ as above, and if it does not intersect any connector we set $\type_{\hyp{h}} (I_m)=0$.
The assumption that $\gridy {C}_{\hyp{h}}$ is admissible amounts to saying that $\sum_{j=1} ^n \type_{\hyp{h}} (I_j) = 0$. Now one can define the function $\alpha_{\bhyp{h}}$ on the endpoints of the partition by $\alpha_{\bhyp{h}} (\xi_m)=\alpha_{\bhyp{h}}(\xi_0) + \sum _{j=1} ^m \type(I_j)$. 

Let $I_m$ be one of the intervals. If $I_m$ does not intersect any connector then let $\alpha_{\bhyp{h}}$ be the constant function $\alpha_{\bhyp{h}} (\xi_{m-1}) = \alpha_{\bhyp{h}} (\xi_{m})$ on $I_m$.
Otherwise, $I_m$ intersects a connector $C$ which is supported on $\hyp{h}$ and $\hyp{k}$. 
Since $\alpha_{\bhyp{h}} (\xi_m)- \alpha_{\bhyp{h}} (\xi_{m-1}) = \type_{\hyp{h}} (I_m)$ there is an orientation $\hs{k}$ of $\hyp{k}$ such that $\alpha_{\bhyp{h}}$ is equal to $\charf{\bhs{k}}$ on $\xi_{m-1},\xi_m$. Thus we can define $\alpha_{\bhyp{h}} = \charf{\bhs{k}}$ on $I_m$.

Clearly this construction satisfies the required condition. Moreover, since the construction of this function is invariant to subdivisions the partition $I_1,\ldots,I_m$, we conclude that it is the unique function that satisfies the required conditions.
\end{proof}

\begin{definition}[Grids and oriented grids] A \emph{grid} $\grid{G}$ is a pair $(\gridy{H},\gridy{C})$ of a finite collection of hyperplanes $\gridy{H}$ and a finite collection of disjoint connectors $\gridy{C}$ which are supported on hyperplanes in $\gridy{H}$, satisfying that for each $\hyp{h}\in \gridy{H}$ the set $\gridy {C}_{\hyp{h}}$ of the connectors in $\gridy {C}$ which are supported on $\hyp{h}$ is admissible. 

An \emph{oriented grid} $\grid{G}=(\gridy{H},\gridy{C},\{\alpha_{\bhyp{h}} | \hyp{h}\in \gridy{H}\})$ is a grid $(\gridy{H},\gridy{C})$ together with a choice of a orientations $\alpha_{\bhyp{h}}$ for each $\hyp{h}\in \gridy{H}$.
\end{definition}

Before we proceed to the construction of the cocycle we give the three main examples of grids that we will use. See Figure \ref{fig: grids} and compare with Remark \ref{rmk: Examples of admissible hyperplanes}.

\begin{figure}
\begin{center}
\def\svgwidth{\textwidth}
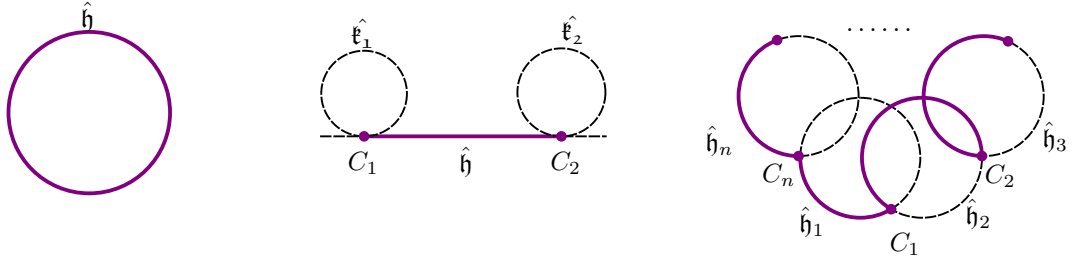
\caption{Examples of grids, from left to right, Example \ref{eg: a hyperplane grid}, Example \ref{eg: an arc grid} and Example \ref{eg: a cycle grid}. The hyperplanes are shown in dotted lines. The connectors and an arbitrary choice of orientation are shown in purple.}
\label{fig: grids}
\end{center}
\end{figure}

\begin{example} [a hyperplane grid]\label{eg: a hyperplane grid}
The grid consisting of one hyperplane $\gridy{H}=\{\hyp{h}\}$ and no connectors. 
\end{example}

\begin{example} [an arc grid]\label{eg: an arc grid}
The grid consisting of three hyperplanes $\gridy{H}=\{\hyp{h},\hyp{k}_1,\hyp{k}_2\}$ and two disjoint connectors $C_1,C_2$ of $\hyp{h}$ with $\hyp{k}_1,\hyp{k}_2$ respectively, such that $\type(C_i,\hyp{h},\hyp{k}_i) = (1,0)$.
\end{example}

\begin{example} [a cycle grid]\label{eg: a cycle grid}
The grid $\gridy{H}$ consisting of cyclically intersecting hyperplanes $\{\hyp{h}_1,\ldots,\hyp{h}_n\}$ and disjoint connectors $\gridy{C} = \{C_1, \ldots ,C_n\}$, such that each $C_i$ is supported on $\hyp{h}_i$ and $\hyp{h}_{i+1}$ (mod $n$), and satisfies $\type(C_i,\hyp{h}_i,\hyp{h}_{i+1})=(1,1)$.
\end{example}

We end this section with the following remark.

\begin{remark}\label{rmk: can partition connectors}
If we replace a connector of a grid by its partition into disjoint clopen sets we obtain a new grid. An orientation for the original grid will remain an orientation for the new. In what follows, this operation will not make any substantial difference, thus we may identify two grids if they have the same hyperplanes and a common partitioning of their connectors.
\end{remark}

\section{Constructing cocycles}
\label{sec: constructing cocycles}

To an oriented grid $\grid{G} = (\gridy{H},\gridy{C},\{\alpha_{\bhyp{h}} | \hyp{h}\in \gridy{H}\})$ we assign a parity function in several steps.

\begin{enumerate}[label=Step \arabic*. , , itemindent = 1.7cm , leftmargin=0cm ]
\item Fix $\CCv{x}_0\in\CC{X}$. 
Let $C\in\gridy{C}$ be a connector supported on $\hyp{h},\hyp{k}$. Since $\hyp{h}\cap\hyp{k}$ is quasi-isometric to a tree, there exists a big enough ball $B(\CCv{x}_0,R)$ in $\CC{X}$ such that $C$ can be written as a finite union of limit sets of the components of $\hyp{h}\cap\hyp{k}\setminus B(\CCv{x}_0,R)$.
Let $R_0>0$ be big enough such that the above is true for all $C\in\gridy{C}$. 

By Remark \ref{rmk: can partition connectors}, we may assume that the connectors are exactly the limit sets of these components.
We denote the component of $\hyp{h}\cap\hyp{k}\setminus B(\CCv{x}_0,R_0)$ that corresponds to $C$ by $\check{C}$.

\item \label{step2}For each $\hyp{h}\in\gridy{H}$ let $\Hs{T}_{\hyp{h}}$ be the set of all halfspaces $\hs{t}\in\Hyp{H}(\hyp{h})$ in the cube complex $\hyp{h}$ (i.e, $\hs{t} = \hyp{h} \cap \hs{h}'$ for some halfspace $\hs{h}'$ in $\CC{X}$) that satisfy:
\begin{itemize}
\item $\hs{t}\cap B(\CCv{x}_0,R_0) = \emptyset$
\item $\bhs{t}$ is contained in a neighborhood $V$ as in Definition \ref{def: orientation of hyp}, and in particular it intersects at most one connector $C$,
\item the halfspace $\hs{t}$ does not intersect any of the sets $\check{C}$ of the connectors in $\gridy{C}_{\hyp{h}}$ except if $\bhs{t}$ intersects a connector $C$ supported on $\hyp{h}$ and $\hyp{k}$ in which case $\hs{t} \cap \hyp{k} = \hs{t} \cap \check{C}$, and in particular  $\hs{t}\cap\check{C}$ is a hyperplane in the cube complex $\hs{t}$. 
\end{itemize}

For every $\xi\in\bhyp{h}$ there exists $\hs{t}\in\Hs{T}_{\hyp{h}}$ such that $\xi\in\obhs{t}$.

Let $R_1>R_0$ be such that $B(\CCv{x}_0,R_1)$ contains $\hyp{h}\setminus \bigcup _{\hs{t}\in\Hs{T}_{\hyp{h}}} \hs{t}$ for all $\hyp{h}\in\gridy{H}$.

\item We extend the function $\alpha_{\bhyp{h}}:\bhyp{h}\setminus \bigcup \gridy {C}_{\hyp{h}} \to \Z/2$ to a function $\alpha_{\hyp{h}}:\hyp{h}^{(0)}\setminus B(\CCv{x}_0,R_1) \to \Z/2$ in the following way: On each halfspace $\hs{t}\in\Hs{T}_{\hyp{h}}$ either $\bhs{t}$ does not intersect any connector, in which case $\alpha_{\bhyp{h}}|_{\bhs{t}}$ is constant, and we define $\alpha_{\hyp{h}}|_{\hs{t}}$ to be the constant function with the same value. Otherwise, $\bhs{t}$ intersects exactly one connector $C$ which is supported on $\hyp{h}$ and $\hyp{k}$, and there exists a halfspace $\hs{k}$ bounded by $\hyp{k}$ such that $\alpha_{\bhyp{h}}|_{\bhs{t}} = \charf{\bhs{k}}|_{\bhs{t}}$. We define $\alpha_{\hyp{h}}|_{\hs{t}} = \charf{\hs{k}}|_{\hs{t}}$. 

Since vertices in a hyperplane correspond to edges in the cube complex, we identify the function $\alpha_{\hyp{h}}$ with the corresponding function on the edges of $\CC{X}\setminus B(\CCv{x}_0,R_1)$ which are transverse to $\hyp{h}$.
\end{enumerate}

\begin{claim}\label{claim: cocycle}
The sum $\alpha = \sum_{\hyp{h}\in \gridy{H}} \alpha_{\hyp{h}}$ is a 1-cocycle on $\CC{X}\setminus B(\CCv{x}_0,R_1)$. 
\end{claim}

\begin{proof}
Let $\CCc{S}$ be a square in $\CC{X}\setminus B(\CCv{x}_0,R_1)$, and let $\hyp{h}$ and $\hyp{k}$ be the two hyperplanes which are transverse to the edges of $\CCc{S}$.
Let $\CCc{e}_{\hyp{h}},\CCc{e}_{\hyp{h}}'$ (resp. $\CCc{e}_{\hyp{k}},\CCc{e}_{\hyp{k}}'$) be the pair of opposite edges of $\CCc{S}$ which are transverse to $\hyp{h}$ (resp. $\hyp{k}$) (see Figure \ref{fig: cocycle}).
We have to show that \[ \alpha(\CCc{e}_{\hyp{h}})+\alpha(\CCc{e}_{\hyp{h}}')+\alpha(\CCc{e}_{\hyp{k}})+\alpha(\CCc{e}_{\hyp{k}}' )=0 \]. In other words, we have to show that if one pair of opposite edges in $\CCc{S}$ has different values of $\alpha$ then so does the other pair. 

\begin{figure}
\begin{center}
\def\svgwidth{0.3\textwidth}
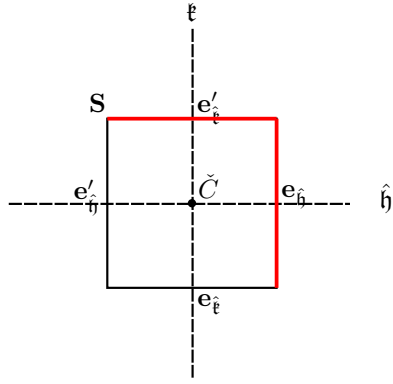
\caption{A square $\CCc{S}$, and the hyperplanes that cross it $\hyp{h},\hyp{k}$. The edges on which the 1-cocycle $\alpha$ obtains the value $1$ are shown in red. If on one pair of opposite edges $\alpha$ obtains different values then so it does on the other pair.}
\label{fig: cocycle}
\end{center}
\end{figure}

Assume that the two opposite edges of $\CCc{S}$ which are transverse to $\hyp{h}$ have different values of $\alpha$, i.e, $\alpha(\CCc{e}_{\hyp{h}})+\alpha(\CCc{e}_{\hyp{h}}')=1$. 
Then, if we view $\CCc{e}_{\hyp{h}},\CCc{e}'_{\hyp{h}}$ as vertices of $\hyp{h}$, there exists a halfspace $\hs{t}\in\Hs{T}_{\hyp{h}}$ of $\hyp{h}$ such that the two vertices are in $\hs{t}$, and there exists a connector $C$ such that these two vertices are separated by the hyperplane $\check{C}\cap\hs{t}$ of $\hs{t}$. 
Thus, the other pair of opposite edges $\CCc{e}_{\hyp{k}},\CCc{e}_{\hyp{k}}'$ of $\CCc{S}$, viewed as an edge in $\hyp{h}$ are transverse to the hyperplane $\check{C}\cap\hs{t}$ of $\hs{t}$.
This implies that $\hyp{k}\in\gridy{H}$ and the connector $C$ which correspond to $\check{C}$ is supported on $\hyp{h},\hyp{k}$. 

Since $\check{C}$ also separates $\CCc{e}_{\hyp{k}},\CCc{e}_{\hyp{k}}'$, now viewed as vertices in $\hyp{k}$, they must also have different values of $\alpha$.
%
\end{proof}

We remark that even though the definition of $\alpha$ depends on $R_0$ and $R_1$, any two such cocycles are equal outside a large enough ball $B(\CCv{x}_0,R)$.

\section{The parity function of a grid}
\label{sec: parity of grid}

Let $\grid{G}=(\gridy{H},\gridy{C},\{\alpha_{\bhyp{h}}|\hyp{h}\in \gridy{H}\})$ be an oriented grid, and let $\alpha$ be the 1-cocycle on $\CC{X}\setminus B(\CCv{x}_0,R_1)$ defined in the previous section. 
Since $\CC{X}$ has trivial first cohomology over $\Z/2$ at infinity we deduce that there is $R>R_1$ such that $\alpha$ is a coboundary in $\CC{X}\setminus B(\CCv{x}_0,R)$. That is, there exists $\pi:\CC{X}^{(0)}\setminus B(\CCv{x}_0,R) \to\Z/2$ such that $\delta\pi=\alpha$. Since $\CC{X}$ is connected at infinity, any two such coboundaries $\pi$ are equal outside a bigger ball and up to adding the constant function $1$. We call the function $\pi$ the parity function of the grid $\grid{G}$.

Let us denote by $\bdry\grid{G}$ the set $\{\xi |\exists\hyp{h}\in \gridy{H}, \alpha_{\bhyp{h}}(\xi)=1\} \cup \bigcup _{C\in\gridy{C}} C \subset \bdry\CC{X}$. For example, in Figure \ref{fig: grids}, the set $\bdry\grid{G}$ is shown in plain purple.

\begin{lemma}\label{lem: parity extends to boundary}
The coboundary $\pi$ extends to a function, which we denote as well by $\pi$, which is defined on $\bdry\CC{X}\setminus \bdry\grid{G}$, and is constant on connected components of $\bdry\CC{X}\setminus\bdry\grid{G}$
\end{lemma}

\begin{proof}
It is enough to prove that any point in $\bdry\CC{X}\setminus \bdry\grid{G}$ has a neighborhood $V$ in $\CC{X}\cup\bdry\CC{X}$ on which $\pi$ is constant.

For a point $\xi \in \bdry\CC{X}\setminus\bdry\grid{G}$ there is a neighborhood halfspace $\hs{t}$ disjoint from $\bdry\grid{G}$ and from $B(\CCv{x}_0,R)$. This halfspace, by definition, does not meet the cocycle $\alpha$ and thus the parity function on $\hs{t}$ is constant.
\end{proof}

Since the parity function is defined uniquely up to constants, it makes more sense to consider the function $\Delta\pi:(\bdry\CC{X}\setminus\bdry\grid{G})^2 \to \Z/2$ (similarly, $\Delta\pi:(\CC{X}^{(0)}\setminus B(\CCv{x}_0,R))^2 \to \Z/2$) given by $\Delta\pi(x,y)=\pi(x)-\pi(y)$. Moreover, for a path $P$ whose endpoints are in the domain of $\pi$ we define $\Delta\pi(P)$ to be the value of $\Delta\pi$ on the pair of endpoints of $P$.

\begin{example}\label{eg: parity of hyperplane grid}
Let $\hyp{h}$ be a hyperplane, and let $\grid{G}=(\gridy{H}=\{\hyp{h}\},\gridy{C}=\emptyset,\alpha_{\bhyp{h}} =1)$ be the hyperplane grid described in Example \ref{eg: a hyperplane grid}. 
Then, $\alpha = \charf{\hyp{h}}$, $\bdry\grid{G}=\bhyp{h}$, and the parity function $\pi = \charf{\hs{h}}$ for some choice of halfspace $\hs{h}$ of $\hyp{h}$, 
which extends to the boundary to the function $\pi = \charf{\bhs{h}}|_{\bdry\CC{X}\setminus\bhyp{h}}$. We denote this parity function by $\pi_{\hyp{h}}$.
Similarly, the function $\Delta\pi_{\hyp{h}}$ is the function that returns $0$ if the two points are on the same side of $\hyp{h}$, and returns $1$ if they are separated by $\hyp{h}$.
\end{example}

\begin{remark} \label{rmk: type and parity}
Using this notation, we can rewrite the notation introduced in \ref{subsec: connectors} as follows. If $C$ is a connector supported on $\hyp{h},\hyp{k}$ and $I_j$ is an interval on $\hyp{h}$ as in the definition of $\type(C)$. Then, \[ \type_{\hyp{h}}(I_j) = \Delta\pi_{\hyp{k}} (I_j) \]
\end{remark}

The following lemma follows from the definitions. 

\begin{lemma}\label{addition of grids} 
Let $\grid{G}_1,\grid{G}_2$ be two oriented grids with disjoint or identical connectors. Then we denote by $\grid{G}_1+\grid{G}_2$ the pair $\grid{G}=(\gridy{H}_1\cup \gridy{H}_2, \gridy{C}_1 \triangle \gridy{C}_2)$. The pair $\grid{G}$ is a grid, and there exists an orientation on $\grid{G}$ which satisfies $\alpha_{\bhyp{h}}=\alpha_{1,\bhyp{h}} + \alpha_{2,\bhyp{h}}$ when all of the functions are defined. 
Moreover, the cocycle (resp. the parity function) associated to $\grid{G}$ is the sum of the cocycles (resp. the parity functions) associated with $\grid{G}_1,\grid{G}_2$.\qed
\end{lemma} 

\begin{lemma}\label{lem: the parity locally}
Let $\xi\in\bdry\grid{G}\setminus \bigcup_{C\in\gridy{C}} C$ and let $\hyp{h}_1,\ldots,\hyp{h}_n\in \gridy{H}$ be all the hyperplanes such that $\xi\in\bhyp{h}_i$ and $\alpha_{\bhyp{h}_i}(\xi)=1$. Then there is a neighborhood $V$ of $\xi$ in $\CC{X}\cup\bdry\CC{X}\setminus\bdry\grid{G}$ on which \[ \Delta\pi = \sum_{i=1} ^n \Delta\pi_{\hyp{h}_i} \]
\end{lemma}

\begin{proof}
Let $\hs{k}$ be a halfspace neighborhood of $\xi$, disjoint from the connectors $\gridy{C}$ and from $B(\CCv{x}_0,R)$ such that for $i=1,\ldots,n$, $\hs{k}\cap\hyp{h}_i \in \Hs{T}_{\hyp{h}_i}$ (see \ref{step2} in Section \ref{sec: constructing cocycles}). Then, since $\alpha_{\bhyp{h}_i} (\xi)=1$ we deduce that $\alpha (\hyp{h} \cap \hs{k})=1$. Thus, on $\hs{k}$ the cocycle $\alpha$ is the sum of the corresponding hyperplane cocycles (see Example \ref{eg: parity of hyperplane grid}), i.e $\alpha = \sum _{i=1} ^n \charf{\hyp{h}_i\cap\hs{k}}$, and the lemma follows.
\end{proof}

By subdividing a path which avoids the connectors of a grid into small enough segments and applying the previous lemma, one can deduce the following.

\begin{corollary}\label{prop: computing the parity}
Let $P$ be a path in $\bdry\CC{X}$ whose endpoints are in $\bdry\CC{X}\setminus\bdry\grid{G}$ and which is disjoint from $\bigcup_{C\in\gridy{C}} C$.
For each $\hyp{h}\in \gridy{H}$ let $J_{\hyp{h},1},\ldots,J_{\hyp{h},n}$ be $n=n(\hyp{h})$ disjoint subintervals of $P$ such that $P\cap\alpha_{\bhyp{h}} ^{-1} (1)\subset \bigcup_{j=1}^n J_{\hyp{h},j}$, the endpoints of each subinterval $J_{\hyp{h},i}$ are disjoint from $\bhyp{h}$ and its interior intersects $\bhyp{h}$ in $\alpha_{\bhyp{h}} ^{-1} (1)$. 
Then \[ \Delta\pi(P) =\sum_{\hyp{h} \in \gridy{H}} \sum_{j=1} ^n \Delta\pi_{\hyp{h}} (J_{\hyp{h},j}). \]\qed
\end{corollary}

\begin{proposition}\label{prop: two state solution}
Under the assumptions of Subsection \ref{subsec: sum of prelims}, there is no connector of type $(1,0)$. 
\end{proposition}

\begin{proof}
Assume that $\hyp{h}$ and $\hyp{k}$ have a connector $C$ such that $\type(C,\hyp{h},\hyp{k})=(1,0)$. By the dynamics of $\Stab(\hyp{h})$ on $\bhyp{h}$ we can find such $\hyp{k}$ and $C$ in any small open set in $\bhyp{h}$.
	
	\begin{figure}[!ht]
		\begin{center}
			\def\svgwidth{\textwidth}
			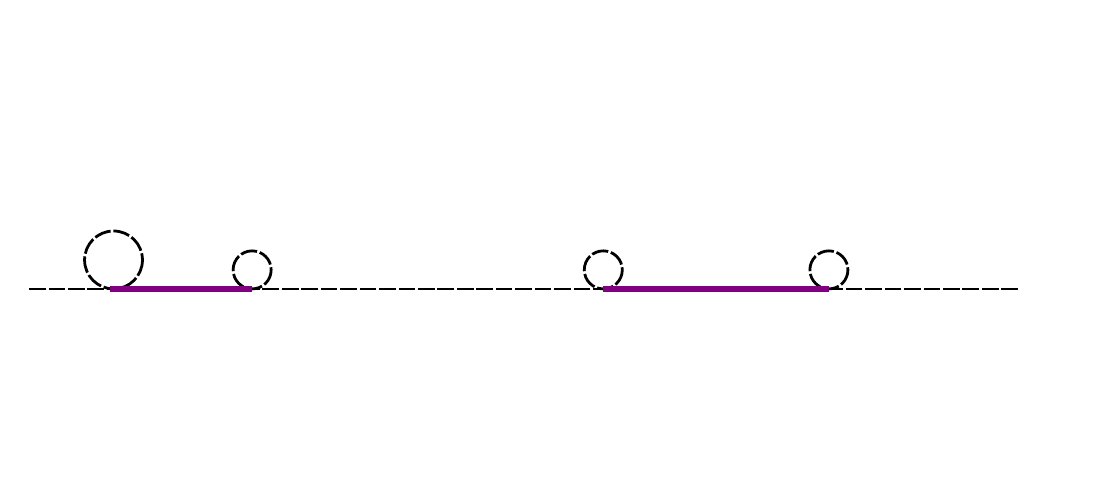
			\caption{The hyperplanes $\hyp{h}$ and $\hyp{k}_i,\hyp{k}'_i$ are drawn as dotted lines. The paths $P_1$ and $P_2$ are shown in red and blue. The boundary of the grid which contradicts the existence of type $(1,0)$ connectors is shown in purple.}
			\label{fig: two state}
		\end{center}
	\end{figure}

Since $\CC{X}$ is assumed to be essential, the open limit sets of  halfspaces $\obhs{h}$ and $\comp{\obhs{h}}$ are non-empty. Let $\xi,\zeta$ be points in $\obhs{h}$ and $\comp{\obhs{h}}$ respectively. By the Cyclic Connectivity Theorem (by Ayres \cite{Ayr29} and Whyburn \cite{Why31}), since our space has no cutpoints we can find two disjoint paths $P_1,P_2$ connecting $\xi$ and $\zeta$ (see Figure \ref{fig: two state}). Let $A_i=\bhyp{h}\cap P_i$, $i=1,2$.

Let $I_1,I_2,\ldots,I_n$ be disjoint intervals of $\hyp{h}$ disjoint from $A_2$ and such that the union of their interiors contain $A_1$. Let $\delta>0$ be such that  $\delta$-neighborhoods of the endpoints of $I_1,\ldots,I_n$ are pairwise disjoint are disjoint from $P_1\cup P_2$. For each $I_i$ find two hyperplanes $\hyp{k}_i,\hyp{k}'_i$ contained in the $\delta$-neighborhoods of the two endpoints of $I_i$, and two connectors $C_i,C'_i$ supported on  $\hyp{h}$ and $\hyp{k}_i,\hyp{k}'_i$ respectively,  with $\type(C_i,\hyp{h},\hyp{k}_i)=\type(C'_i,\hyp{h},\hyp{k}'_i)=(1,0)$, as described above, in each such open neighborhood.

The grid $(\gridy{H}=\{\hyp{h}_1,\hyp{k}_1,\hyp{k}'_1\ldots \hyp{k}_n,\hyp{k}'_n\},\gridy{C}=\{C_1,C'_1,\ldots,C_n,C'_n\})$ is the sum of arc grids described in the examples of grids, Example \ref{eg: an arc grid}. Orient the grid such that $\alpha_{\bhyp{h}}$ is $1$ on $A_1$ (and $0$ on $A_2$).

By Corollary \ref{prop: computing the parity} applied to $P_1$ (with $n(\hyp{h})=1, J_{\hyp{h},1}=P_1$ and $n(\hyp{k}_i)=n(\hyp{k}'_i)=0$ for $i=1,\ldots,n$), we get that $\pi(\xi)-\pi(\zeta) = \Delta\pi(P_1) = \Delta\pi_{\hyp{h}} (P_1) = 1$ because $P_1$ intersects the grid exactly as it would intersect the hyperplane grid $\hyp{h}$, and we recall that $\xi$ and $\zeta$ are on different sides of $\hyp{h}$ thus have different parity.

On the other hand $P_2$ does not intersect the grid, and thus $\pi(\xi)-\pi(\zeta) = \Delta\pi(P_2) = 0$. This contradicts the existence of connectors of type $(1,0)$.
\end{proof}

The following useful corollary is a direct consequence of Proposition  \ref{prop: two state solution}.
\begin{corollary}\label{cor: two states solution}
Given two hyperplanes $\hyp{h} $and $\hyp{k}$. Assume that there exists two points of $\bhyp{h}$ on different sides of $\bhyp{k}$. Then the hyperplanes $\hyp{h}$ and $\hyp{k}$ share a $(1,1)$ connector. \qed
\end{corollary}

\section{Approximations of curves and arcs}
\label{sec: PH curves}

Let $J$ be a Jordan curve or arc.

In this section, we assume that the Jordan arc $J$ is parametrized by $[0,1]$. Given a parameter $t$ the associated point is $J(t)$, conversely if $x$ is a point on $J$ its parameter will be denoted $\inv{x}$. A segment between two points $\zeta$ and $\xi$ on the arc $J$ is denoted $[\zeta,\xi]_J$. 

\begin{definition}\label{def: PH curve}
A \emph{piecewise-hyperplane curve} or \emph{PH curve} (resp. \emph{PH arc})  is a (not necessarily simple) parametrized curve (resp. arc) $P$ on $\bdry\CC{X}$ which has a partition into finitely many segments $I_1,\ldots,I_n$ such that:
\begin{enumerate}
\item \label{point1} each segment $I_i$ is a segments of the boundary $\bhyp{h}$ of some hyperplane $\hyp{h}$,
\item\label{point2} any two segments on each hyperplane $\hyp{h}$ are disjoint,
\item \label{point3}any two consecutive hyperplane segments on $\bhyp{h},\bhyp{k}$ are connected along a limit point of their intersection, called a \emph{vertex} of the curve,  and the pair $(\hyp{h},\hyp{k})$ supports a $(1,1)$ connector.
\end{enumerate} 

An $\epsilon$-\emph{approximating PH curve} (resp. arc) is a PH curve (resp. arc) which is at distance less than $\epsilon$ from $J$ with respect to the sup metric.

\end{definition}

The following lemmas show that given a Jordan curve one can construct PH approximations of the curve with certain technical restrictions that will be useful later on. We first prove it in the case of an arc.
A set $\Hs{H}$ of limit sets of halfspaces is a \emph{$\delta$-cover} of a set $S$ in the boundary, if it is a cover of $S$ and if every element in $\Hs{H}$ has diameter less than $\delta$.


\begin{lemma}\label{lem: existence of PH approximation of arc}
Let $J$ be a Jordan arc such that the two endpoints $\zeta_0$ and $\zeta_1$ of $J$ are on the boundaries of the hyperplanes $\bhyp{k}_0$ and $\bhyp{k}_1$ respectively. Let $\hyp{h}$ be a hyperplane such that $\bhyp{h}$ does not intersect $J$ except maybe for it endpoints. Let $\Hs{K}$ be a finite set of hyperplanes. For every $\epsilon>0$, there exists $\delta$ such that for any two distinct halfspaces $\obhs{h}_0$ and $\obhs{h}_1$ not in $\Hs{K}$ of diameter less than $\delta$ covering $\zeta_0$ and $\zeta_1$ respectively, there exists a $\delta$-cover $\Hs{H} = \left\{ \obhs{h}_i\right\}$ of $J$ and an $\epsilon$-approximating PH arc supported on $\cup \bhyp{h_i}$, such that
\begin{itemize}
\item the set $\Hs{H}$ is disjoint from $\Hs{K}$,
\item the PH arc has at most one segment on each $\bhyp{h_i}$,
\item the endpoints of the PH arc lie on $\bhyp{k}_0$ and $\bhyp{k}_1$,
\item the first and last segments of the PH arc are supported on $\bhyp{h}_0$ and $\bhyp{h}_1$ respectively,
\item the hyperplanes $\hyp{h}_0$ and $\hyp{h}_1$  share a $(1,1)$ connector with $\hyp{k}_0$ and $\hyp{k}_1$ respectively,
\item the PH arc may intersect $\bhyp{h}$ only on $\bhyp{h}_0$ and $\bhyp{h}_1$.
\end{itemize}
\end{lemma}

\begin{proof}

Let $0<\delta<\epsilon/3$ be such that for any two points $\zeta$ and $\xi$ on $J$ at distance less than $\delta$, the segment $[\zeta,\xi]_J$ has diameter less than $\epsilon/2$, and such that  for $i=0,1$ there exists a point on $\bhyp{k}_i$ at distance greater than $\delta$ from $\zeta_i$. 

Let $\obhs{h}_0$ and $\obhs{h}_1$ be open limit sets of halfspaces of diameter less than $\delta$ covering $\zeta_0$ and $\zeta_1$ respectively. Note that by definition of $\delta$. there are points of $\bhyp{k_i}$ in both sides of $\bhyp{h}_i$, which by Corollary \ref{cor: two states solution} insures that $\bhyp{h}_i$ and $\bhyp{k}_i$ share a $(1,1)$ connector.

Let $\delta'<\delta$ be such that $J\setminus \left( \obhs{h}_1 \cup \obhs{h}_1 \right)$ is at distance more that $\delta'$ from $\bhyp{h}$, $\zeta_0$ and $\zeta_1$, and such that no halfspace in $\Hs{K}$ has diameter less that $\delta'$.

Let $\Hs{H}' = \left\{\obhs{h}_2,\dots\obhs{h}_n\right\}$ be a minimal finite $\delta'$-cover of $J\setminus \left( \obhs{h}_1 \cup \obhs{h}_1 \right)$.
The $\Hs{H} = \left\{\obhs{h}_0,\dots\obhs{h}_n\right\}$ is a minimal finite $\delta$-cover of $J$ such that 
\begin{itemize}
\item the set $\Hs{H}$ disjoint from $\Hs{K}$,
\item the only limit set  containing $\zeta_0$ is $\obhs{h}_0$,
\item the only limit set containing $\zeta_1$ is $\obhs{h}_1$,
\item the only limit sets that may intersect $\bhyp{h}$ are $\obhs{h}_0$ and $\obhs{h}_1$.
\end{itemize}

Let $\delta''$ be smaller than the diameter of any $\bhyp{h_i}$.

For $i=0,\ldots n$ define $y_i = \sup J^{-1}(\obhs{h}_i) \in [0, 1]$.
We define a successor function $S$ as follow: $S(i)=j$ where $J(y_i) \in \obhs{h}_j$ and $y_j$ is maximal under this condition. 
Note that the function is define unless $i=1$ (and $y_i=1$), and that, when it is defined, $y_{S(i)} > y_i$.
These conditions insure that there exists $p$ such that $\zeta_2 \in \obhs{h}_{S^p(0)} (= \obhs{h}_1$).


Let $v_i$ be an intersection point of $\bhyp{h}_{S^{i-1}(0)}$ and $\bhyp{h}_{S^{i}(0)}$. 
For all $0<i\leq p$, 
take distinct halfspaces $\obhs{h}'_i$ covering $v_i$ and of diameter less than $\delta''$. 
The parameter $\delta''$ is small enough to insure that they do not belong to $\Hs{K}$, do not intersect $\bhyp{h}$ and do not contain $\zeta_0$ or $\zeta_1$.
Moreover by choice of $\delta''$, we can find points $\obhs{h}_{S^{i-1}(0)}$ (resp. $\obhs{h}_{S^{i}(0)}$) on both sides of $\bhyp{h}'_i$. Thus applying Corollary \ref{cor: two states solution}  $\bhyp{h}'_i$ shares $(1,1)$ connectors with both $\obhs{h}_{S^{i-1}(0)}$ and $\obhs{h}_{S^{i}(0)}$.
 Denote $v'_i$ and $v''_i$ intersections of $\bhyp{h}'_i$ with $\bhyp{h}_{S^{i-1}(0)}$ and $\bhyp{h}_{S^{i}(0)}$ respectively.

Let $v''_0$ (resp. $v'_{p+1}$) be an intersection point of $\bhyp{h}_0$ and $\bhyp{k}_0$ (resp. $\bhyp{h}_1$ and $\bhyp{k}_1$) which is closest to $\zeta_0$ (resp. $\zeta_1$). 

Let $I_i$  be one of the two intervals of $\bhyp{h}_{S^i(0)}$ with endpoints $v''_i$ and $v'_{i+1}$.
Let $I'_i$ be one of the two intervals of $\bhyp{h}'_i$ with endpoints $v'_i$ and $v''_i$.

Let $\eta$ be such that paths of $J$ parametrized by $[x-\eta, x + \eta]$ have length less than $\delta$ and that for all $i$ for which $S(i)$ is defined, we have $y_S(i)-y_{i}> 2\eta$.

We obtained a PH curve $P = (I_0, I'_1,I_1,\dots, I'_p, I_p)$ that we parametrize continuously such that $v'_{i}$ and $v''_{i}$ have parameter $y_{S^{i-1}(0)}-\eta$ and $y_{S^{i_1}(0)}+\eta$ respectively (with the convention that the parameters of $v''_0$ is $0$ and the one of $v'_{p+1}$ is 1).

It remains to show that the path $P$ is at distance $\epsilon$ from $J$ with respect to the sup metric. 
Points in $I_i$ and $J(y_{S^{i}(0)})$ are in  $\bhs{h}_{S^{i-1}(0)}$ and are thus at distance less than $\delta$.
The points $J(y_{S^{i-1}(0)})$ and $J(y_{S^{i}(0)})$ belong to $\obhs{h}_{S^i(0)}$ and thus are at a distance less than $\delta$, which by the definition of $\delta$ implies that the path $J([y_{S^{i-1}(0)},y_{S^{i}(0)}])$ connecting them on $J$ has diameter less than $\epsilon/2$, and thus also $J([y_{S^{i-1}(0)}+\eta,y_{S^{i}(0)}-\eta])$. 
Hence $I_i$ and $J([y_{S^{i-1}(0)}+\eta,y_{S^{i}(0)}-\eta])$ are at distance less than $\epsilon/2 + 3\delta< \epsilon$.

Similarly $I'_i$ is at distance less than $2\delta$ from $J(y_{S^{i-1}(0)})$, and so $J([y_{S^{i-1}(0)}-\eta,y_{S^{i-1}(0)}+\eta])$ are at distance less than $3\delta< \epsilon$.
This completes the proof that $P$ is an $\epsilon$-approximation of $J$ that satisfies the requirements of the lemma.
\end{proof}

In what follows we will denote by $[x,y]_{\hyp{h}}$ one of the two subsegments of the limit set of a hyperplane $\bhyp{h}$ that connect $x$ and $y$ (which are not necessarily distinct).

\begin{definition}
Given a path $J$ and a hyperplane $\bhyp{h}$, a \emph{$\delta$-bypass} of $x\in J\cap \bhyp{h}$ on $\bhyp{h}$ is a segment $[x,y]_{\hyp{h}}$ of $\bhyp{h}$ with $y\in J\cap \bhyp{h}$ such that $[x,y]_{\hyp{h}}$ can be partitioned to a finite union $\bigcup_{i=0}^{n-1} [x^i,x^{i+1}]_{\hyp{h}}$ (with $x^0=x, x^n=y$) of disjoint segments (except at their extremities) of diameter less than $\delta$ , with extremities in $J\cap \bhyp{h}$ and such that for every element $\zeta$ of $[x^i,x^{i+1})\cap J$, we have $\inv{\zeta} < \inv{x^{i+1}}$.
A $\delta$-bypass $[x,y]_{\hyp{h}}$ is \emph{maximal} if $y$ has maximal parameter amongst all $\delta$-bypasses of $x$ (i.e, if $[x,y']_{\hyp{h}}$ is another $\delta$-bypass then $J^{-1}(y')\le J^{-1}(y)$). 
A maximal bypass is \emph{degenerate} if $x=y$.
\end{definition}

\begin{remark}\label{remarkbypass}
\begin{enumerate}
\item \label{remark1} If $[x,y]_{\hyp{h}}$ is a $\delta$-bypass, then $\forall \zeta \in[x,y)\cap J,~\inv{\zeta}<\inv{y}$.
\item \label{remark2} If $[x,y]_{\hyp{h}}$ is a maximal $\delta$-bypass, then for any $z\in [x,y]_{\hyp{h}}\cap J$, the segment $[z,y]_{\hyp{h}}$ is a maximal $\delta$-bypass.
\item \label{remark3} There exists $\nu$ such that for all maximal $\delta$-bypass $[x,y]_{\hyp{h}}$, the segment of $J$ parametrized by $(\inv{y},\inv{y} + \nu)$ does not intersect $\bhyp{h}$.
\item \label{remark4} If two maximal $\delta$-bypasses $[x,y]_{\hyp{h}}$ and $[x',y']_{\hyp{h}}$ intersect, then $y=y'$. Indeed, otherwise we can suppose that $\inv{y'}<\inv{y}$. From point \ref{remark1}, the element $y$ cannot belong to $[x',y']_{\hyp{h}}$, so $[x,y]_{\hyp{h}} \not \subset [x',y']_{\hyp{h}}$. Therefore either $x'$ or $y'$ belongs to $[x,y]_{\hyp{h}}$, using point \ref{remark2}, we get a contradiction.
\item \label{remark5}For every $x\in J\cap \bhyp{h}$ there exists a maximal $\delta$-bypass. Indeed, the set of end points $y$ of  $\delta$-bypasses $[x,y]_{\hyp{h}}$ is closed.
\end{enumerate}
\end{remark}

\begin{definition}
A \emph{detour} of $\delta$-bypasses is a set of maximal bypasses $\left\{[x_i,y_i]_{\hyp{h}} \right\}$, such that for any two bypasses $[x_i,y_i]_{\hyp{h}}$ and $[x_j,y_j]_{\hyp{h}}$, either $\inv{y_i} < \inv{x_j}$ or $\inv{y_j}<\inv{x_i}$. 

A detour is \emph{covering}, if for any element $z\in J\cap \bhyp{h}$, there is one bypass $[x,y]_{\hyp{h}}$ of the detour, such that $\inv{x} \leq \inv{z} \leq \inv{y}$. Or equivalently, $J\cap\bhyp{h} \subseteq \bigcup _i [x_i,y_i]_J$.

From Remark \ref{remarkbypass} point \ref{remark4}, bypasses of a detour are disjoints.
\end{definition}

\begin{lemma}
There exists a finite covering detour.
\end{lemma}

\begin{proof}
We construct the detour by induction: take the element $x$ of smallest parameter that is not covered by the detour, and by Remark \ref{remarkbypass} point \ref{remark5} add a maximal bypass of $x$. Remark \ref{remarkbypass} point \ref{remark3} implies that the process finishes in a finite number of steps. 
\end{proof}

%

\begin{lemma}\label{lem: existence of PH approximation for arc with h}
Let $J$ be an arc such the two endpoints $\zeta_0$ and $\zeta_1$ of $J$ are on the boundaries of hyperplanes $\bhyp{k}_0$ and $\bhyp{k}_1$ respectively. Let $\epsilon>0$ and let $\hyp{h}$ be a hyperplane. There exists an $\epsilon$-approximating PH arc which satisfies the following two  conditions:
\begin{itemize}
\item the hyperplanes containing the first and last segments of the PH arc share a $(1,1)$ connector with respectively $\hyp{k}_0$ and $\hyp{k}_1$,
\item any intersection of $P$ with $\bhyp{h}$ is along a segment of $P\cap\bhyp{h}$.
\end{itemize}
\end{lemma}

\begin{proof}
Let $\delta_1 < \epsilon/9$ such that for any two points $\zeta$ and $\xi$ on $J$ at distance less than $\delta_1$ the segment $[\zeta,\xi]_J$ has diameter less than  $\epsilon/3$.
This insure that a $\delta_1$-bypass $[x,y]_{\hyp{h}}$ is an $\epsilon/3+\delta_1$-approximation of the segment $[x,y]_J$ (with the natural parametrization of $[x,y]_{\hyp{h}}$ by $[\inv{x},\inv{y}]$ which agrees with $J$ on the endpoints $x^i\in J\cap \bhyp{h}$ of its subdivision as it appears in the definition of a $\delta$-bypass). 
Indeed, $[x,y]_{\hyp{h}}$ is a union of segments $[x^i,x^{i+1}]_{\hyp{h}}$ of diameter less than $\delta_1$ and with extremities on $J$  with increasing  parameter. 
By the choice of $\delta_1$, the paths $[x^i,x^{i+1}]_J$ have diameter less than $\epsilon/3$. 
Since the parameters are increasing, the segments $[x^i,x^{i+1}]_J$ on $J$ are disjoint (except the extremities), and each $[x^i,x^{i+1}]_{\hyp{h}}$ is an $\epsilon/3+\delta_1$-approximation of $[x^i,x^{i+1}]_J$.

Let $\mathcal D = \{[x_1,y_1]_{\hyp{h}},\ldots,[x_n,y_n]_{\hyp{h}}\}$ be a finite covering detour of $\delta_1$-bypasses, ordered by their parameter along the curve $J$.
Let $\delta_2<\delta_1$ be such that each bypass is at distance at least $2\delta_2$ from any other.
Let $\delta_3<\delta_2$ be such that each pair of points $\zeta$ and $\xi$ on $\bhyp{h}$ at distance less than $\delta_3$ has a path from $\zeta$ to $\xi$ on $\bhyp{h}$ with diameter less than  $\delta_2$.

Let $J_i$ be the segment $[y_i,x_{i+1}]_J$ on $J$ for all $i=0,\ldots,n$ (where $y_0$ and $x_{n+1}$ are the two endpoints of $J$,  $\zeta_0$ and $\zeta_1$, respectively). 
Each of the endpoints of $J_i$ are on $\bhyp{h}$ except maybe for $J_0$ and $J_n$ which are on $\bhyp{k}_0$ and $\bhyp{k}_1$. 

Let $\delta < \delta_3$ be obtained from Lemma \ref{lem: existence of PH approximation of arc} for all the $J_i$.
Take $\obhs{h}_{i,x}$ and $\obhs{h}_{i,y}$ covers of respectively $x_i$ and $y_i$ of diameter less than $\delta$, such that the limit sets of their bounding hyperplanes are all disjoints. 
It may happen that $x_i=y_i$ if a bypass is degenerate, but by taking $\obhs{h}_{i,y}$ small enough, we can ensure that its boundary does not intersect that of $\obhs{h}_{i,x}$.
Since $\mathcal D$ is covering, the interior of each $J_i$ does not intersect $\bhyp{h}$. 

Using Lemma \ref{lem: existence of PH approximation of arc}, we have $\epsilon$-approximation $P_i$ for $J_i$ stating and ending with segments on $\obhs{h}_{i,y}$ and $\obhs{h}_{i+1,x}$. We can also insure that the $\epsilon$-approximation $P_i$ are supported on distinct hyperplanes. Indeed we can construct the $P_i$'s one after the other and add all the used hyperplanes to the set $\Hs{K}$ of Lemma \ref{lem: existence of PH approximation of arc}.

Let $\zeta_i$ and $\xi_i$ be the starting and ending points of $P_i$. By construction, the point $\zeta_i$ (resp. $\xi_i$) is at distance less than $\delta$ from $y_{i-1}$ (resp. $x_i$). Thus there is a segment $[\xi_i,\zeta_{i+1}]_{\hyp{h}}$ on $\bhyp{h}$ at distance less than $\delta_2$ from
 $[x_i, y_{i}]_{\hyp{h}}$, and by the choice of $\delta_2$ these paths are disjoint and not reduced to a point (since $\xi_i\neq\zeta_{i+1}$). Thus $[\xi_i,\zeta_{i+1}]_{\hyp{h}}$ is a $\epsilon/3+\delta_1 + \delta_2<\epsilon$-approximation of $[x_i,y_i]_{\hyp{h}}$. By concatenating alternatively the $P_i$ and the $[\xi_i,\zeta_{i+1}]_{\hyp{h}}$, we obtain an $\epsilon$-approximating PH arc of $J$ satisfying the required properties.
\end{proof}

\begin{lemma}\label{lem: existence of PH approximation}
Let $J$ be Jordan curve, let $\epsilon>0$, and let $\hyp{h}$ be a hyperplane. 
Then, there exists an $\epsilon$-approximating PH curve such that any intersection of $P$ with $\bhyp{h}$ is along a segment of $P$.
\end{lemma}

\begin{proof}
Up to taking $\epsilon$ small enough we can assume that the diameter of $J$ is larger than $\epsilon$.

To approximate a curve, we do the following: 
if $J = \bhyp{h}$, then the approximation is $\bhyp{h}$, otherwise there is a point $\xi \in J \setminus \bhyp{h}$. 
Let $\eta<\epsilon/2$ be such that for any two points $\zeta$ and $\xi$ on $J$ at distance less than $\eta$, one of the segments of $J$ with extremities $\zeta$ and $\xi$ has diameter less than $\epsilon/2$.

Take a halfspace $\hs{k}$ such that $\obhs{k}$ contains $\xi$, it has diameter less than $\eta$ and such that $\bhyp{k}$ does not intersect $\bhyp{h}$. This is possible since $\xi$ is not on $\bhyp{h}$. Let $\zeta_0$ and $\zeta_1$ be to points of $J\cap \bhyp{k}$. Let $J_0$ be the path between $\zeta_0$ and $\zeta_1$ of diameter less than $\epsilon/2$, and let $J_1$ be the complementary path. Notice that any segment of $\bhyp{k}$ is an $\epsilon$-approximation of $J_0$.
Applying Lemma \ref{lem: existence of PH approximation for arc with h} to $J_1$ (and the hyperplanes $\hs{k}_1=\hs{k}_2=\hs{k}$), let $P_1$ be a $\delta$-approximation of $J_1$ with $0<\delta<\epsilon$ which is less than the diameter of $\hs{k}$. 
The approximation $P_1$ of $J_1$ begins and ends on $\hs{k}$ and thus can be closed by a segment to form an approximation of $J$. Notice that \ref{point2} of Definition \ref{def: PH curve} is insure by the fact that $\hs{k}$ cannot be used in $P_1$, the point \ref{point3}  follows from Lemma \ref{lem: existence of PH approximation for arc with h}. The condition on intersection with $\bhyp{h}$ follows from conclusion of Lemma \ref{lem: existence of PH approximation for arc with h} and the fact that $\bhyp{k}$ does not intersect $\bhyp{h}$.
\end{proof}

\section{The parity function of a Jordan curve} 
\label{sec: the parity of a Jordan curve}

\subsection{The parity function of a PH curve}
\label{subsec: the parity of a PH curve}
Let $P$ be a PH curve. We assign a parity function to $P$ in the following way.
Let $\delta>0$ be smaller than the half the minimal distance between the vertices of $P$. 
For each vertex $\xi$ of $P$ which is supported on the two hyperplanes $\hyp{h}$ and $\hyp{k}$, let $C$ be a type $(1,1)$ connector which is small enough such that in each of $\bhyp{h},\bhyp{k}$ the connector is contained in an interval which stays in the $\delta$-neighborhood of $\xi$ (one can find such a connector using the dynamics of $\Stab_G (\hyp{h}\cap\hyp{k})$ on $\bhyp{h}\cap\bhyp{k}$). 
Let $\grid{G}$ be the grid consisting of the hyperplanes on which $P$ is supported and the connectors which we assigned to each vertex of $P$. 
We choose the orientation such that outside of the $\delta$-neighborhoods of the vertices of $P$, $P=\bdry\grid{G}$ (one can do so, since each connector is of type $(1,1)$ and contained in an interval which stays in the $\delta$-neighborhood of the corresponding vertex).

\begin{lemma}\label{lem: parity of PH curve}
For $\epsilon>0$ small enough, there exists $\delta>0$ such that outside an $\epsilon$-neighborhood of the vertices of $P$ and outside $P$, the parity function of $\grid{G}$ (associated to $\delta$) does not depend on the choice of connectors (up to constants).
\end{lemma}

\begin{proof}
By Corollary \ref{cor: no local cutpoints}, for any $\epsilon>0$ there exists $\delta>0$ small enough such that any two points in $\bdry\CC{X}$ outside the $\epsilon$-neighborhood of the vertices of $P$ can be connected with a path that stays outside the $\delta$-neighborhood of the vertices of $P$.
We note that this path satisfies the conditions of Proposition \ref{prop: computing the parity} and since it does not enter the $\delta$-neighborhood of the vertices of $P$ it does not depend on the choice of connectors.
\end{proof}

We deduce the following
\begin{corollary}
For a PH curve $P$ the parity $\Delta\pi$ of the grid constructed above for $\delta$ has a limit as $\delta\to 0$ which is defined outside $P$.\qed
\end{corollary}

\subsection{The parity function of a Jordan curve}

We would like to follow the same idea, approximating the Jordan curve $J$ with PH curves, and taking the limit of their parity functions as the definition of the parity function of $J$.

\begin{lemma}\label{lem: parity of Jordan curve}
Let $J$ be a Jordan curve. Let $\eta >0 $, there exists $\epsilon>0$ such that for any two  $\epsilon$-approximating PH curves $P,\tilde{P}$ for $J$ the parity function of $P$ and $\tilde{P}$ are equal (up to constants) outside an $\eta$-neighborhood of $J$.
\end{lemma}

\begin{proof}
By the No-Blob Lemma (Lemma \ref{lem: no blob}) let $\eta'$ be such that outside any set that is contained in a ball $B(\xi,\eta')$ there is at most one component of which is not contained in $B(\xi,\eta)$. 
By local connectedness, let $0<\epsilon<\frac{\eta'}{4}$ be such that any two points in $B(\xi,\epsilon)$ can be connected by a path in $B(\xi,\frac{\eta'}{2})$.

Let $P$ and $\tilde{P}$ be two $\epsilon$-approximating PH curves for $J$. Let $\epsilon'>0$ be smaller than the diameter of the hyperplanes of $P$ and $\tilde{P}$, and let $P'$ be an $\epsilon'$ approximating PH curve for $J$ such that the diameter of the  hyperplanes of $P'$ are smaller than $\epsilon'$. 

It suffices to show that the parity functions of $P'$ and $P$ (similarly, $\tilde{P}$) coincide (up to constant) outside the $\eta$-neighborhood of $J$.

\begin{figure}[t!]
\begin{center}
\def\svgwidth{\textwidth}
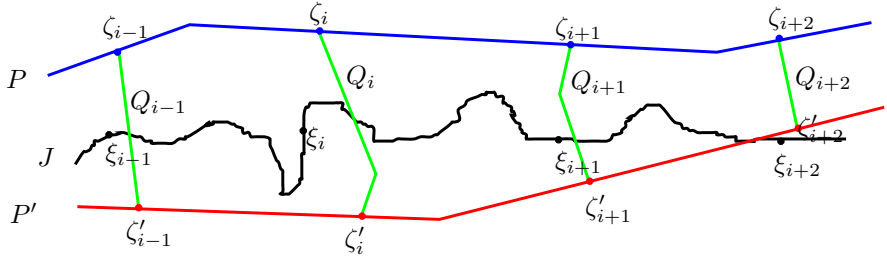
\caption{The Jordan curve $J$ (in black), two approximating PH curves $P$ (in blue) and $P'$ (in red) and the auxiliary PH curves $Q_i$ (in green).}
\label{fig: parity for Jordan}
\end{center}
\end{figure}

Let $\xi_0,\xi_1,\ldots,\xi_n=\xi_0$ be a partition of $J$ into intervals of diameter less than $\epsilon$. Let $\zeta_i$ and $\zeta'_i$ be the corresponding points on $P$ and $P'$ at distance less than $\epsilon$ from $\xi_i$. By perturbing these points if necessary, we may assume that $\{\zeta_i, \zeta'_i\}_{i=0} ^{n-1}$ are $2n$ distinct points, and that they are not vertices of $P$ and $P'$ (see Figure \ref{fig: parity for Jordan}).
Since $\epsilon<\frac{\eta'}{4}$, the interval between $\zeta_i$ and $\zeta_{i+1}$ on $P$ (resp. $\zeta'_i$ and $\zeta'_{i+1}$ on $P'$) is contained in $B(\xi_i,\frac{\eta'}{2})$.


By Lemma \ref{lem: existence of PH approximation} and the definition of $\epsilon$, connect $\zeta_i,\zeta'_i$ by a PH arc $Q_i$ in $B(\xi,\frac{\eta'}{2})$.
Form the short closed path $P_i$ by connecting $\zeta_i,\zeta_{i+1},\zeta'_i,\zeta'_{i+1}$ along $P,Q_{i+1},P',Q_i$.
We note that by the construction, each $P_i$ is contained in $B(\xi_i,\frac{\eta'}{2})$.

Let $\delta<\eta'/2$ be such that the $\delta$-neighborhoods of all the vertices of both $P,P'$ and the arcs $Q_i$ are disjoint. 
Choose a connector in the $\delta$-neighborhood of each vertex (including connectors for the endpoints of $Q_i$ with $P$ and $P'$). 
Let $\grid{G},\grid{G}'$ and $\grid{G}_i$ be the grids described in Subsection \ref{subsec: the parity of a PH curve} for $P,P'$ and $P_i$.
Note that by construction $\grid{G}+\grid{G}'=\sum_i \grid{G}_i$ (in the notation of Lemma \ref{addition of grids}).

Since each $P_i$ is contained in $B(\xi_i,\frac{\eta'}{2})$ and $\delta<\frac{\eta'}{2}$, the boundary $\bdry\grid{G}_i$ of the corresponding grid is contained in $B(\xi_i,\eta')$. This implies that the complement of each $\bdry\grid{G}_i$ has at most one component which is not contained in the $\eta$-neighborhood of $J$. Thus, the parity function $\pi_i$ of $\grid{G}_i$ is constant outside the $\eta$-neighborhood of $J$. If we denote by $\pi$ and $\pi'$ the parity functions of $P$ and $P'$ respectively, then by Lemma \ref{addition of grids} $\pi-\pi'=\sum_i \pi_i$. We deduce that $\pi - \pi'$ is constant outside an $\eta$-neighborhood of $J$.
\end{proof}

This implies the following.

\begin{corollary}
For a Jordan curve, the parity $\Delta\pi$ for an $\epsilon$ approximating PH curve has a limit as $\epsilon \to 0$, which is defined outside $J$.\qed
\end{corollary}
\section{Jordan's theorem}
\label{sec: Jordan thm}

\begin{proposition}\label{prop: parity around a vertex}
Let $P$ be a PH curve, and let $\xi_i$ be a vertex of $P$. Let $J_{i},J_{i+1}$ be the segments of $P$ on $\bhyp{h}_i,\bhyp{h}_{i+1}$ respectively which are incident to $\xi_i$. 
Assume $J_i$ does not intersect $\bhyp{h}_{i+1}$ except at $\xi_i$.
Let $I_{i}$ be an open interval in $\bhyp{h}_{i+1}$ which is disjoint from $P$ and such that one of its endpoints is $\xi_i$. 
Then points on $I_i$ have the same parity with respect to $P$ as any point in $\comp{\hs{h}}_{i+1}$ which is close enough to $J_{i+1}$, where $\hs{h}_{i+1}$ is the halfspace of $\hyp{h}_{i+1}$ which contains $J_i$.
\end{proposition}

\begin{proof}
We denote $\xi = \xi_i, J'=J_i, J=J_{i+1}, I=I_i, \hyp{h}'=\hyp{h}_i, \hyp{h}=\hyp{h}_{i+1}$ (see Figure \ref{fig: two state for PH curves}).

\begin{figure}
\begin{center}
\def\svgwidth{0.7\textwidth}
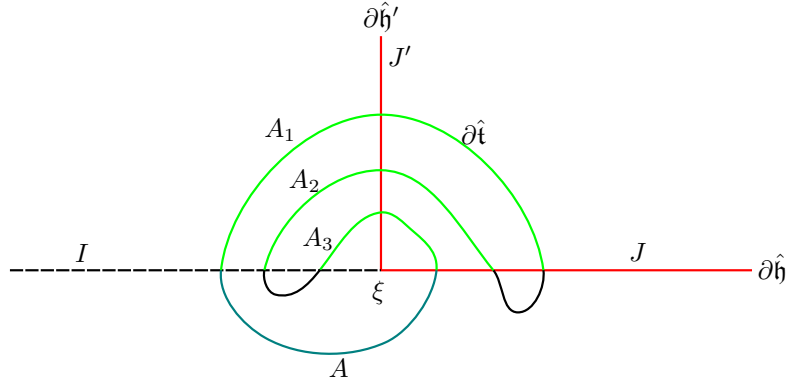
\caption{The two possible cases of the arcs of $\bhyp{t}$ drawn on the same figure. On the bottom, an arc $A$ which does not cross $J'$ and connects $I$ and $J$. On the top, a collection of arcs $A_i$ which cross $J'$}
\label{fig: two state for PH curves}
\end{center}
\end{figure}

Clearly all the points on $I$ have the same parity, because they are connected with a path that does not meet $P$, namely a subinterval of $I$.
By Lemma \ref{lem: the parity locally}, around any point of $J$ different of $\xi$, the parity function is the characteristic function of $\hs{h}$ or $\comp{\hs{h}}$. Thus, all the points in $\comp{\hs{h}}$ which are close enough to $J$ have the same parity.

Thus it suffices to find two points, one on $I$ and the other in $\comp{\hs{h}}$ close enough to $J$ that have the same parity.

Let $\hs{t}$ be a halfspace neighborhood of $\xi$ which is small enough such that $\hs{t}\cap (P\cup\bhyp{h}) \subset I\cup J \cup J'$. 
The connected components of $\bhyp{t}\setminus \bhyp{h}$ are arcs which stay on one side of $\hyp{h}$.
If one of these arcs $A$ connects $I$ with $J$ and is contained in $\comp{\hs{h}}$ then since it does not cross $P$ we deduce that its endpoint on $I$ has the same parity as any other point on $A$, and in particular points which are arbitrarily close to $J$ in $\comp{\hs{h}}$.

If not, let us analyse how the arcs of $\bhyp{t}\setminus \bhyp{h}'$ intersect $J'$. First we observe that only finitely many arcs of $\bhyp{t}\setminus \bhyp{h}$ intersect $J'$. We denote them $A_1,\ldots,A_n$. 

Let $\pi$ be the parity function of $P$, and let $\pi_{\hyp{t}}$ be the parity function of the single hyperplane grid $\hyp{t}$, i.e,  $\pi_{\hyp{t}}$ is (up to a constant) the characteristic function of $\bhs{t}$. 

If we follow the segment $J'$ we see that one of its endpoints is in $\bhs{t}$ and the other is in $\comp{\bhs{t}}$. Therefore the difference in the value of $\pi_{\hyp{t}}$ between the two endpoints is $1$. On the other hand, by Corollary \ref{prop: computing the parity}, it is the sum over the types of intersections of $J'$ with the hyperplane $\hyp{t}$. We break this sum to each arc $A_i$.

\[ 1=\Delta\pi_{\hyp{t}}(J') = \sum_i \type _{\hyp{h}'} (A_i \cap J') = \sum_i \type _{\hyp{t}}(A_i \cap J') = \sum_i \Delta\pi_{\hyp{h}'} (A_i)\]

Where the third equality follows from Proposition \ref{prop: two state solution}. Note that any arc $A_i$ whose endpoints lie on the same side of $\xi$, namely, both on $I$ or both on $J$, contributes $0$ to the sum because its endpoints have the same parity with respect to $\pi$.

Therefore, there exists $i$ such that the arc $A_i$ whose endpoints are on $I$ and $J$ and such that $\Delta\pi_{\hyp{h}'} (A_i \cap J') = 1$. Since $A_i$ intersects $P$ only in $J'$ and is contained in $\bhs{h}$ it follows that points on $I$ have different parity than points in $\bhs{h}$ which are close enough to $J$, from which the desired conclusion follows.
\end{proof}

\begin{corollary}\label{cor: two state solution for PH curves}
Let $P$ be a PH curve, and let $\xi_{i},\xi_{i+1}$ be two consecutive vertices of $P$. 
Let $J_{i}, J_{i+1}, J_{i+2}$ be the segments of $P$ on $\bhyp{h}_i,\bhyp{h}_{i+1},\bhyp{h}_{i+1}$ respectively which are incident with $\xi_i$ and $\xi_{i+1}$ in the obvious way. 
Assume $J_i$ and $J_{i+2}$  do not intersect $\bhyp{h}_{i+1}$ except at $\xi_i$ and $\xi_{i+1}$ respectively, and $P$ intersects $J_{i+1}$ only at the segment $J_{i+1}$.
Let $I_i$ (resp. $I_{i+2}$) be an open interval on $\bhyp{h}_{i+1}$ which is disjoint from $P$ and one of its endpoints is $\xi_i$ (resp. $\xi_{i+1}$)

Then, $J_i$ and $J_{i+2}$ are on the same side of $\hyp{h}_{i+1}$ if and only if $I_i$ and $I_{i+2}$ have the same parity with respect to $P$.\qed
\end{corollary}

In other words, if we denote by $\Delta\pi (J_{i+1})$ the difference of parities between $I_i$ and $I_{i+2}$ with respect to $P$ and by $\Delta\pi_{\hyp{h}} (J_{i+1})$ the difference of parities between $J_i$ and $J_{i+2}$ with respect to the parity function $\pi_{\hyp{h}}$ of the grid defined by $\hyp{h}$ (see Example \ref{eg: parity of hyperplane grid}), then the previous corollary shows that $\Delta\pi (J_{i+1})=\Delta\pi_{\hyp{h}} (J_{i+1})$.

\begin{theorem} \label{thm: Jordan thm}
Let $J$ be a Jordan curve. There are two points on which the associated parity function $\pi_J$ takes different values. In particular, $J$ separates $\bdry\CC{X}$ into more than two components.
\end{theorem}

\begin{proof}
Let $\zeta_1,\zeta_2$ be two distinct points on $J$, and let $A,B$ be the two arcs on $J$ that connect $\zeta_1$ to $\zeta_2$. Let $\hyp{h}$ be a hyperplane that separates $\zeta_1$ and $\zeta_2$.

Let $r$ be the distance between $A \cap \bhyp{h}$ and $B \cap \bhyp{h}$, and let $\epsilon < \frac{r}{2}$. By Lemma \ref{lem: existence of PH approximation}, let $P$ be an approximating PH curve for $J$ such that the parity functions $\pi_J$ and $\pi$ of $J$ and $P$ respectively coincide outside an $\epsilon$-neighborhood of $J$ or $P$, and such that $P$ intersects $\hyp{h}$ along segments of $P$. It suffices to find two points on $\bhyp{h}$ which are at distance $\epsilon$ from $P$ with different parity with respect to $\pi$. By abuse of notation we denote by $A$ and $B$ the corresponding approximating arcs on $P$.

We partition $\bhyp{h}$ into intervals $A_1,B_1,A_2,B_2,\ldots A_m,B_m$  such that the endpoints of each $A_i,B_i$ are at distance $\epsilon$ from $P$, and each $A_i$ intersects $J$ only in $A$ and $B_i$ intersects $J$ only in $B$. 

We denote by $\pi_{\hyp{h}}$ the parity function defined by the hyperplane $\hyp{h}$. The difference $\Delta\pi_{\hyp{h}}(A)=\Delta\pi_{\hyp{h}}(B)=\pi_{\hyp{h}}(\zeta_1)-\pi_{\hyp{h}}(\zeta_2) =1 $ since $\zeta_1,\zeta_2$ are on different sides of $\hyp{h}$. On the other hand, we can compute it using Corollary \ref{prop: computing the parity} and the notation introduced above, and get 
\[ 1=\Delta\pi_{\hyp{h}}(A)=\sum _{j=1} ^m \sum _{J_i\subseteq A_j} \Delta\pi_{\hyp{h}} (J_i) \]  
where the second sum runs over all segments $J_i$ of $P$ which are contained in $A_j$.

By Corollary \ref{cor: two state solution for PH curves}, we can replace the sum by  
\[ 1=\sum _{j=1} ^m \sum _{J_i\subseteq A_j} \Delta\pi (J_i) = \sum _{j=1} ^m \Delta\pi (A_j) \]
where $\Delta\pi (A_j)$ denotes the difference in the parity of the two endpoints of $A_j$ with respect to $\pi$.

From this it follows that one of the intervals $A_j$ has $\Delta\pi (A_j)=1$, and thus its endpoints have different parity, as desired.
\end{proof}

\bibliographystyle{plain}
\bibliography{main}

\end{document}